\documentclass[a4paper,12pt,reqno]{amsart}
\usepackage{mathrsfs,calligra,latexsym, bm, amscd,amssymb,url}
\usepackage{xcolor, mathtools}
\PassOptionsToPackage{pdfusetitle,pagebackref,colorlinks}{hyperref}
\usepackage[all,cmtip]{xy}
\usepackage{bookmark}
\hypersetup{
linkcolor={red!70!black},
citecolor={green!70!black},
urlcolor={teal!80!black},
linkcolor={teal!80!black}
}
\usepackage[all,cmtip]{xy}  
\usepackage{graphicx}
\usepackage{color}
\usepackage{hyperref}
\usepackage{enumerate}
\usepackage{stmaryrd}

\usepackage[T1]{fontenc}         
\usepackage{ae}									
\usepackage[all,cmtip]{xy}  
\usepackage{graphicx}
\usepackage{color}

\newtheorem{theorem}{Theorem}[section]

\newtheorem{lemma}[theorem]{Lemma}
\newtheorem{proposition}[theorem]{Proposition}
\newtheorem{corollary}[theorem]{Corollary}

\newtheorem{conjecture}[theorem]{Conjecture}
\newtheorem{example}[theorem]{Example}

\theoremstyle{question}

\theoremstyle{definition}
\newtheorem{definition}[theorem]{Definition}

\theoremstyle{remark}
\newtheorem{remark}[theorem]{Remark}

\newtheoremstyle{cited}{.5\baselineskip\@plus.2\baselineskip\@minus.2\baselineskip}{.5\baselineskip\@plus.2\baselineskip\@minus.2\baselineskip}{\itshape}{}{\bfseries}{\bfseries .}{5pt plus 1pt minus 1pt}{\thmname{#1}\thmnumber{~#2}\thmnote{ \normalfont#3}}
\theoremstyle{cited}

\newtheoremstyle{citeddef}{.5\baselineskip\@plus.2\baselineskip\@minus.2\baselineskip}{.5\baselineskip\@plus.2\baselineskip\@minus.2\baselineskip}{}{}{\bfseries}{\bfseries .}{5pt plus 1pt minus 1pt}{\thmname{#1}\thmnumber{~#2}\thmnote{ \normalfont#3}}
\theoremstyle{citeddef}

\usepackage{tikz-cd}
\usepackage{enumitem}

\PassOptionsToPackage{pdfusetitle,pagebackref,colorlinks}{hyperref}
\usepackage{bookmark}
\hypersetup{
linkcolor={red!70!black},
citecolor={green!70!black},
urlcolor={blue!80!black}
}

\def\CC{{\mathbb C}}

\def\PP{{\mathbb P}}

\def\HH{{\mathbb H}}

\def\llra{\hbox to 10mm{\rightarrowfill}}

\def\lllra{\hbox to 15mm{\rightarrowfill}}

\def\phi{{\varphi}}

\def\cI{\mathcal{I}}

\def\cL{\mathcal{L}}
\def\cO{\mathcal{O}}
\def\cG{\mathcal{G}}

\def\cH{\mathcal{H}}

\def\cS{\mathcal{S}}

\def\cN{\mathcal{N}}

\def\cV{\mathcal{V}}

\let\tilde\widetilde

\DeclareMathOperator{\alb}{alb}

\DeclareMathOperator{\codim}{codim}

\DeclareMathOperator{\Dom}{Dom}

\DeclareMathOperator{\Exc}{Exc}
\DeclareMathOperator{\Hom}{Hom}

\DeclareMathOperator{\Supp}{Supp}

\begin{document}

\title[Geometry of Holomorphic One-forms on Smooth Projective Varieties]{Geometry of Holomorphic One-forms on Smooth Projective Varieties}
\author{Jiabin Du}
\address{Chern Institute of Mathematics and LPMC\\ Nankai University\\Tianjin 300071\\ China}
\email{jiabin.du@nankai.edu.cn}

\author{Feng Hao}
\address{School of Mathematics\\ Shandong University\\ Jinan 250100\\ China}
\email{feng.hao@sdu.edu.cn}

\author{Haoyuan Li}
\address{School of Mathematics\\ Shandong University\\ Jinan 250100\\ China}
\email{haoyuan.li@mail.sdu.edu.cn}

\author{Zichang Wang}
\address{School of Mathematical Science\\ University of Science and Technology of China\\ Hefei 230026\\ China}
\email{wangzichang@mail.ustc.edu.cn}

\date{\today}



\begin{abstract}
Let $f\colon X\to A$ be a morphism from a smooth projective variety to an abelian variety. We show that if an irreducible component of the exceptional set of a birational map from $X$ to its good minimal model has general type image in $A$, then every 1-form pulled back from $A$ has a zero. For the Albanese map, this yields a sharp dimension bound for spaces of nowhere vanishing holomorphic 1-forms, which refines the Popa-Schnell inequality.

Furthermore, we investigate the structure of the spaces of holomorphic 1-forms with zeros, and show that they are linear for large classes of varieties. Also, we construct a delicate example of a smooth projective subvariety of an abelian variety for which the holomorphic 1-forms with positive-dimensional zero loci do not form a linear subset. Finally, we study algebraic surfaces admitting holomorphic 1-forms that have zeros and do not arise from cohomology jump loci.

%In addition, we investigate the linearity of the spaces of degenerate holomorphic 1-forms, construct an explicit example of a smooth subvariety of an abelian variety with a non-linear lower level set and analyze the structure of algebraic surfaces admitting degenerate holomorphic 1-forms  not arising from cohomology jump loci. 
\end{abstract}
\maketitle
\vspace*{6pt}
%\tableofcontents  % for this guide only.
% A table of contents should normally not be included
\section{Introduction}

Holomorphic 1-form is a basic concept in algebraic geometry and occupies a central position in analyzing the geometry and topology of  irregular varieties. The space of holomorphic 1-forms governs the morphisms from varieties to abelian varieties, especially the Albanese maps, so the information encoded in holomorphic 1-forms should potentially reveal many structural properties of irregular varieties (see e.g., \cite{CCH25}, \cite{Chu26b}, \cite{CKSW26}, \cite{DHL24}, \cite{HK05}, \cite{Hao24}, \cite{HS21}, \cite{Kot22}, \cite{LZ05},  \cite{Pie25}, \cite{PS14}, \cite{Sch21}, \cite{SY25}). It is usually very difficult to keep track of the changes of zeros of holomorphic 1-forms under algebraic operations on varieties like birational modifications, branched covers, base changes, and  hyperplane sections.  

For the first part of this article, we study how exceptional set of a birational modification can give rise to zeros of holomorphic 1-forms. 

\begin{theorem} [{=Theorem \ref{thm:birat-zero-holomorphic 1-forms}}]\label{thm:main theorem 1.1}
Let $f\colon X\to A$ be a morphism from a smooth projective variety $X$ to an abelian variety $A$. Suppose $X$ admits a good minimal model $X^{\text{min}}$ together with a birational map $\mu\colon X\dashrightarrow X^{\text{min}}$, and  there is an irreducible component $Z$ of the exceptional set $\textnormal{Exc}(\mu)$ of $\mu$ such that $f(Z)$ is of general type. Then  for every holomorphic 1-form $\omega\in H^0(A, \Omega_A^1)$, $f^*\omega$ has zero.
\end{theorem}

The celebrated result of Popa-Schnell \cite{PS14} shows that every holomorphic 1-form on a smooth projective variety of general type has a zero. Theorem \ref{thm:main theorem 1.1} actually gives a second geometric source of zeros arising from birational modification. More generally, for any vector space $H$ whose nonzero elements are holomorphic 1-forms without zero, \cite{PS14} gives the Popa-Schnell inequality $\dim H \leq \dim X-\kappa(X)$, which bounds the dimension of $H$ by birational  invariants of $X$. However, $H$ is very sensitive to birational modifications of
$X$. As a direct consequence of Theorem \ref{thm:main theorem 1.1}, our following Theorem \ref{Thm: Theorem 1.2} makes up the information of the upper bound
of $\dim H$ arising from birational modifications.
 
\begin{theorem}[{=Theorem \ref{thm:compare with PS}}] \label{Thm: Theorem 1.2}
Let $X$  be a smooth projective variety with Albanese morphism $a\colon X\to A$. Suppose $X$ admits a good minimal model $X^{\text{min}}$ together with a birational map $\mu\colon X\dashrightarrow X^{\text{min}}$ such that $\textnormal{Exc}(\mu)$ is nonempty.  Then for any irreducible component $Z$ of $\textnormal{Exc}(\mu)$, and any vector subspace $H\subseteq H^0(X, \Omega_X^1)$ whose nonzero elements are nowhere vanishing holomorphic 1-forms, we have $$\dim H\leq \min\{\dim a(Z)-\kappa(a(Z)),\  \dim X-\kappa(X)\}.$$  Moreover, we also have $\dim Z\geq \dim H+1$, and both inequalities are sharp.
\end{theorem}

Theorem \ref{Thm: Theorem 1.2} also indicates that the size of the exceptional sets of birational modifications is controlled by holomorphic 1-forms without zeros from below. One should also note that either of the two upper bounds $\dim a(Z)-\kappa(a(Z))$ and $\dim X-\kappa(X)$ can be strictly smaller than the other (see Remark \ref{rmk:compair-inequalities}).

Another consequence (see Corollary \ref{coro:simple-nowhere vanishing 1-form}) of Theorem \ref{thm:main theorem 1.1} shows that a morphism $f\colon X\to A$ from a smooth projective variety $X$ to a simple abelian variety $A$ is smooth provided that the pluricanonical system of general fibres $F$ has small base locus and there exists a nowhere vanishing holomorphic 1-form pullback from $A$.   The rough idea of the proof of Theorem \ref{thm:main theorem 1.1} can be summarized as follows. Suppose by contradiction that $f^*\omega$ has no zero. Based on Popa-Schnell \cite{PS14} and Hao-Wang-Zhang \cite{HWZ26} (see also Church \cite{Chu26b}), we may assume $X^{\text{min}}\to A$ splits into a product after passing to a finite \'etale cover of $X^{\text{min}}$. The translation vector fields of this product preserve the higher codimensional exceptional components near their general points. Consequently, the images of these components via $f$ are invariant under translation by positive dimensional abelian subvarieties, hence $Z$ must be of codimension one provided that $f(Z)$ is of general type. Next we show that $Z$ is covered by curves $C$ intersecting negatively with $Z$. On the one hand, this gives positivity of the restriction of the rank-one conormal sheaf $\mathcal{N}^*_{Z/X}$ on $C$. On the other hand, the assumption $f^*\omega$ has no zero would provide a nowhere vanishing section of $\mathcal{N}^*_{Z/X}|_C$, which is a contradiction.

In the second part of this article, we aim to describe the structure of the space of holomorphic 1-forms
with zeros on $X$ and its lower level sets:
\[
W^i(X):=\{\omega\in H^0(X, \Omega_X^1)\ |\ Z(\omega)\neq \emptyset,\ \codim Z(\omega)\leq i \},
\] 
where $Z(\omega)$ is the set-theoretic zero loci of $\omega$. These lower level sets, introduced by Budur-Wang-Yoon \cite{BWY16}, are all algebraic cones nailed to the origin of $H^0(X,\Omega_X^1)$. Each lower level set $W^i(X)$ contains \textit{linear subsets}\footnote{We call a subset of a complex vector space a \textit{linear subset}, if it is a union of finitely many vector subspaces.} $V^i(X)$ (see Section \ref{sect:linearity-non-formal} for the definition) consisting of holomorphic 1-forms arising from cohomology jump loci. The linearity of $V^i(X)$ is essentially coming from the formality of the variety $X$. However, the structure of $W^i(X)$ is much more delicate and encodes much finer algebraic nature of the variety $X$. For the structure of the whole space of holomorphic 1-forms with zeros, we prove

\begin{proposition}[{=Propositions \ref{thm:linearity-zeros-gmm} and \ref{prop:23linearity}}]\label{prop:gmm-sav-linear}
    Let $X$ be a smooth projective variety of dimension $n$. If either
     $X$ is a good minimal model or  $n\leq 3$,  then the space $W^n(X)$ of holomorphic 1-forms with zeros is a linear subset.
\end{proposition}

 In general, one could expect that at least the birational modification would make $W^n(X)$ rather delicate and out of control. %The structure of $W^n(X)$ relies on a concrete description of the birational operations involved in the minimal model program from the smooth projective variety $X$ to its minimal models.
As for the lower level sets $W^i(X)$ ($i<n$), Wang-Zhang \cite{WZ25} provides an example of a threefold $X$ arising from a careful blowup of a three-dimensional subvariety $Y$ of an abelian fourfold along a general curve in $Y$ such that $W^2(X)$ is not linear. However, for smooth projective subvarieties $X$ of abelian varieties, all the lower level sets $W^i(X)$ were expected to be linear by experts. Surprisingly, we construct an explicit example of a smooth projective fourfold $X$, which is a subvariety of an abelian variety, such that $W^3(X)$ is not linear.

\begin{theorem}[{=Theorem \ref{thm:non-linear example}}]\label{main-thm:nonlinear-ex}
There exists a four-dimensional smooth projective subvariety $X$ of an abelian sixfold, such that the third lower level set $W^3(X)$ is not linear.
\end{theorem}

Theorem \ref{main-thm:nonlinear-ex} answers a question \cite[Question 1.10]{DHL24}. Analyzing the non-linearity of $W^3(X)$ is rather delicate, but the construction of $X$ in Theorem \ref{main-thm:nonlinear-ex} can be easily described as follows. Let $A$ be a simple abelian 5-fold and $E$ be an elliptic curve. Take a general hyperplane section $Y$ of $A$ and a general smooth curve $C$ in $Y$. We carefully choose a hyperplane section $X$ of $Y\times E$ that contains $C\times E$ such that $X$ is smooth and in a nice position. Then $X$ is a smooth subvariety of $A\times E$, and one can show $W^3(X)$ is not linear. 

The example in Theorem \ref{main-thm:nonlinear-ex} shows that the holomorphic 1-forms with zeros that are not from cohomology jump loci are mysterious and should reflect deep algebro-geometric nature of varieties. In the last part of this article, we provide a cohomological criterion for this kind of holomorphic 1-forms, which is easily deduced from \cite{GL87}. Also, we study smooth projective surfaces admitting holomorphic 1-forms that have zeros and do not come from cohomology jump loci.
\subsection{ Notation and Conventions.}

Throughout the paper, we work over the complex number field $\CC$. We follow the standard definitions and notation in \cite{KM98}.
%All varieties are assumed to be projective. 

Let $\phi \colon X\dashrightarrow Y$ be a rational map. The \textit{graph} of $\phi$ is defined as
the  Zariski closure of $\{ (x,\phi(x))| x\in \textnormal{Dom}(\phi) \}$ in $X\times Y$, where $\Dom(\phi)$ is the largest open subset of $X$ over which $\phi$ is regular.

%A rational map $\phi \colon X\dashrightarrow Y$ is called birational if there exist open subsets $U\subseteq X$ and $V\subseteq Y$ such that $\phi\colon U\to V$ is an isomorphism of varieties. 

Let $\phi \colon X\dashrightarrow Y$ be a birational map and $U$  be the largest open subset of $X$ such that $\phi\colon U\cong V$ for some open subset $V\subseteq Y$. Then $X-U$ is called the \textit{exceptional set} of $\phi$, denoted by $\Exc(\phi)$.

For a variety $Z$, the set of all non-singular points of $Z$ is denoted by $Z_{\text{reg}}$. Its complement $Z_{\text{sing}}:=Z-Z_{\text{reg}}$ is called the singular locus of $Z$.

For a subvariety $Z$ of an abelian variety $A$, the \textit{conormal variety} $T^*_{Z}A$ is defined to be the closure of the conormal bundle of the smooth locus $Z_{\text{reg}}$ of $Z$ in $T^*A$. More precisely, $$T_Z^*A\coloneqq \overline{\{(a, \omega)\in Z_{\text{reg}}\times H^0(A, \Omega_A^1)\ |\ \omega|_{Z_{\text{reg}}}(a)=0\}}\subset A\times H^0(A,\Omega_A^1).$$

\medskip
\noindent \textbf{Acknowledgments} We would like to thank Nero Budur, Yongqiang Liu, Christian Schnell, Stefan Schreieder, and Botong Wang for very useful conversations and comments and thank  Shanghai Institute for Mathematics and Interdisciplinary Sciences (SIMIS) for providing an excellent research environment for discussions related to the last section of this work. We also thank Christian Schnell and Akiyoshi Sannai for pointing out some mistakes in the first version. The first named author is supported by NSFC (No. 12601081) and Nankai Zhide Foundation.  The second named author is supported by NSFC (No. 12401054) and SDNSFC (No. 2024HWYQ-009, No. ZR2024MA007, No. tsqn202312060).

\section{Holomorphic 1-forms and birational geometry of algebraic varieties}\label{sect:smooth-mor-simpleAV-1-form}

In this section, we provide a fine study on the behavior of zeros of holomorphic 1-forms under birational modifications.

\subsection{Zeros of holomorphic 1-forms and birational modifications}

\begin{theorem}\label{thm:birat-zero-holomorphic 1-forms}
Let $f\colon X\to A$ be a morphism from a smooth projective variety $X$ to an abelian variety $A$. Suppose $X$ admits a good minimal model $X^{\text{min}}$ together with a birational map $\mu\colon X\dashrightarrow X^{\text{min}}$, and  there is an irreducible component $Z$ of the exceptional set $\textnormal{Exc}(\mu)$ of $\mu$ such that $f(Z)$ is of general type, then  for every holomorphic 1-form $\omega\in H^0(A, \Omega_A^1)$, $f^*\omega$ has zero.
\end{theorem}

 Before the proof of Theorem \ref{thm:birat-zero-holomorphic 1-forms}, we introduce the following lemmas. First of all, recall a concept of ``degenerate'' divisor introduced by S. Takayama (See \cite[Definition~2.4]{Tak08}, also \cite[Definition 2.8]{Lai11}).
\begin{definition}\label{def:degn-f}
 Let $f \colon X \to Y$ be a proper surjective morphism between normal varieties, and $D$ be an effective Weil $\mathbb{R}$-divisor on $X$. We call $D$ a \textit{degenerate divisor with respect to $f$} if $D$ satisfies one of the following two conditions
 
  1) $\codim(\textnormal{Supp}(f(D))) \geq 2$,

  2) $\codim(\textnormal{Supp}(f(D))) = 1$  and there exists a prime divisor
$E\not\subseteq \textnormal{Supp}(D)$ such that $f(E) \subseteq \textnormal{Supp}(f(D))$ has codimension one in $Y$.
\end{definition}

\begin{lemma}[{\cite[Proposition 3.2]{DHL24}}]\label{prop:DHL}
   Let $A$ be an abelian variety and $Z$ be a proper irreducible subvariety of $A$. Then the following are equivalent
   
a) $Z$ is not fibred by tori and $\dim Z > 0$;

b) For a general holomorphic 1-form $\omega\in H^0(A,\Omega^1_A)$, the restriction $\omega|_{Z_{\textnormal{reg}}}$ admits isolated
zeros on the smooth locus $Z_{\textnormal{reg}}$. 

 In particular, let $B \subseteq A$ be the largest (in dimensional sense) abelian subvariety such that $Z$ is fibred by $B$, we have $p_2(T_Z^*A) = H^0(A/B,\Omega^1_{A/B})$, where $T_Z^*A$ is the conormal variety  of $Z$ in $A$, and $p_2\colon T^*A\to H^0(A, \Omega_A^1)$ is the natural projection.
\end{lemma}

% With the notation in Theorem \ref{thm:birat-zero-holomorphic 1-forms}, we have the following lemma.

\begin{proposition}\label{lem:divisorial case}
Let $f\colon X\to A$ be a morphism from a smooth projective variety $X$ to an abelian variety $A$. Suppose $X$ admits a minimal model $X^{\text{min}}$ together with a birational map $\mu\colon X\dashrightarrow X^{\text{min}}$, and  there is a codimension-one irreducible component $Z$ of the exceptional set $\textnormal{Exc}(\mu)$ of $\mu$ such that $f(Z)$ is of general type, then  for every holomorphic 1-form $\omega\in H^0(A, \Omega_A^1)$, $f^*\omega$ has zero.

%Let $f\colon X\to A$ be a morphism from a smooth projective variety  $X$ to an abelian variety $A$. Suppose there is a good minimal model $Y$ such that $f$ is birational to the natural projection $p\colon Y\times A\to A$, i.e., there is a birational map $\mu\colon X\dashrightarrow Y\times A$ such that the following diagram of rational maps is commutative 
%\[\xymatrix{
%X\ar@{-->}[rr]^-{\mu}\ar[dr]_f&&Y\times A\ar[ld]^-p \\
%&A.
%}\] If there is a codimension one irreducible component $Z$ of the exceptional locus $\textnormal{Exc}(\mu)$ of $\mu$ such that $f(Z)$ is of general type, then for every holomorphic 1-form $\omega\in H^0(A, \Omega_A^1)$, $f^*\omega$ has zero.
\end{proposition}

\begin{proof}
Consider the following commutative diagram \[\xymatrix{
X^{\text{min}}\ar[rd]_-{g^{\text{min}}}\ar@{-->}[r]^-{\mu^{-1}}& X\ar[r]^-{f}\ar@{->>}[d]_-{g}&A \\
&Y \ar[ru]_-{\rho},
}\]  where $g$ is the Stein factorization of $f$,   $Y$ is a normal projective variety, and $g^{\text{min}}$ is the composite $g\circ \mu^{-1}$, which is a morphism, since $X^{\text{min}}$ has rational singularities and $\rho$ is a finite morphism to an abelian variety. Assume $g(Z)=Y$, then $f(Z)=\rho(Y)$ is of general type in $A$. Now applying Lemma \ref{prop:DHL}, we get that for every holomorphic 1-form $\omega\in H^0(A, \Omega_A^1)$, $f^*\omega$ has zero. From now on we assume $g(Z)\neq Y$, and prove the proposition in several steps.

\textbf{Step 1. Claim that $Z$ is a degenerate divisor with respect to $g$.}

 By Definition \ref{def:degn-f}, we may assume $g(Z)$ is of codimension one in $Y$. Then one can take an irreducible divisorial component $E$ of ${g^{\text{min}}}^{-1}(g(Z))$ in $X^{\text{min}}$ such that $g^{\text{min}}(E)=g(Z)$. Since $X^{\text{min}}$ is a minimal model and $X$ is smooth, $\mu^{-1}$ does not contract any divisor \cite[Proposition~4.6]{PS2014}. %Consider the inverse birational map $$\eta \colon X\dashrightarrow Y\times A.$$ Applying the negativity lemma of contraction {\color{red}(Add reference BCHM Lemma 3.5.3)} to $\eta$, we have that $\eta$ is $\textnormal{Exc}(\eta)$-negative.  
Therefore, $\mu^{-1}_*E$ is a prime divisor on $X$. Since $g^{\text{min}}\circ \mu=g$, $$g(\mu^{-1}_*E)= g(Z).$$ Also, by normality, $\mu$ induces an isomorphism between Zariski open neighborhoods of the generic points of $E$ and $\mu^{-1}_*E$. In particular, we have $$\mu^{-1}_*E\not\subseteq \textnormal{Exc}(\mu).$$ Since $Z$ is an irreducible component of $\textnormal{Exc}(\mu)$, $$Z\neq \mu_*^{-1}E.$$ Hence $Z$ is degenerate with respect to $g$.

%Thus $Y\times f(Z)$ is a prime divisor in $Y\times A$. Since $Y$ is a good minimal model, so is $Y\times A$. Moreover, since $X$ is smooth, $\mu^{-1}$ does not contract any divisor \cite[Proposition~4.6]{PS2014}. %Consider the inverse birational map $$\eta \colon X\dashrightarrow Y\times A.$$ Applying the negativity lemma of contraction {\color{red}(Add reference BCHM Lemma 3.5.3)} to $\eta$, we have that $\eta$ is $\textnormal{Exc}(\eta)$-negative.  
%Therefore, $\mu^{-1}_*(Y\times f(Z))$ is a prime divisor on $X$. Since $p\circ \mu=f$, $$f(\mu^{-1}_*(Y\times f(Z)))=f(Z).$$ Also, by normality of $X^{\text{min}}$, $\mu$ induces an isomorphism between Zariski open neighborhoods of the generic points of $Y\times f(Z)$ and $\mu^{-1}_*(Y\times f(Z))$. In particular, $\mu^{-1}_*(Y\times f(Z))\not\subseteq \textnormal{Exc}(\mu)$. Since $Z$ is an irreducible component of $\textnormal{Exc}(\mu)$, $$Z\neq \mu_*^{-1}(Y\times f(Z)).$$ Hence $Z$ is degenerate with respect to $f$.

\textbf{Step 2. Claim that there exists an irreducible parameter space $T$, a flat family of projective integral curves $\pi\colon \mathcal{C}\to T$ and an evaluation map $\text{ev}\colon \mathcal{C}\to Z$ such that }

\noindent \textbf{1) $\text{ev}(\mathcal{C})$ is dense in $Z$;}

\noindent \textbf{2) for all $t\in T$, $C_t\coloneqq \text{ev}(\mathcal{C}_t)$ is contracted by $g$ (hence by $f$) and $Z\cdot C_t<0$. }

By Step 1, $Z$ is degenerate with respect to $g$. Applying \cite[Lemma 2.10]{Lai11} and its proof, we have that $Z$ is covered by a bounded family of projective integral curves in $X$ that are contracted by $g$ and intersect $Z$ negatively. One may choose a $g$-very ample line bundle on $X$, with respect to which one has a Hilbert polynomial $P(m)$ and the corresponding Hilbert scheme $\text{Hilb}^{P(m)}_{Z/Y}$ such that the universal evaluation map $\text{ev}$ of $\text{Hilb}^{P(m)}_{Z/Y}$ dominates $Z$. Now we can choose an open dense subset $T$ of an irreducible component of $\text{Hilb}^{P(m)}_{Z/Y}$, together with a flat family of projective integral curves $\pi\colon \mathcal{C}\to T$ induced by the universal family of $\text{Hilb}^{P(m)}_{Z/Y}$, so that $\text{ev}(\mathcal{C})$ is dense in $Z$, and for all $t\in T$, $C_t= \text{ev}(\mathcal{C}_t)$ is contracted by $g$ and $Z\cdot C_t<0$.

\textbf{Step 3. Set $W=f(Z)\subsetneq A$. Claim that there is a nonempty Zariski open subset $W^o$ of the regular locus $W_{\text{reg}}$ of $W$ such that for any closed point $a\in W^o$ there exists a projective integral curve $C_a\subset Z\cap f^{-1}(a)$ such that $$Z\cdot C_a<0; \ C_a\not\subset Z_{\text{sing}}.$$}

By Step 1 and Step 2, we have the following commutative diagram  
\[\xymatrix{
\mathcal C\ar@{->>}[r]^-{\pi}\ar[d]_-{\text{ev}}& T\ar[d]^-{s} \ar[rd]^-{h} \\
Z\ar[r]^{g|_Z}&Y\ar[r]^-{\rho}&A,
}\] where $s\colon T\to Y$ is the structure map $T\subset \text{Hilb}^{P(m)}_{Z/Y}\to Y$. Hence $\text{ev}$ maps each fibre of $\pi$ into a fibre of $g|_Z$. Since $\text{ev}(\mathcal C)$ is dense in $Z$, %$h(T)$ is dense in $W=f(Z)$. In particular 
$h(T)$ contains a nonempty Zariski open subset of $W$, and the incidence variety $$\{t\in T\ |\ C_t\not\subset Z_{\text{sing}}\}$$ is a Zariski dense open subset of $T$. Hence we can take a nonempty Zariski dense open subset $W^o$ of $W_{\text{reg}}$ satisfying the above required properties.

%{\color{red}ADD THE FOLLOWING NOTATION IN THE INTRODUCTION: Now consider the  conormal variety $T_W^*A$ of $W$ in the cotangent bundle $T^*A=A\times H^0(A,\Omega_A^1)$, more precisely $$T_W^*A\coloneqq \overline{\{(a, \omega)\in W_{\text{reg}}\times H^0(A, \Omega_A^1)\ |\ \omega|_{W_{\text{reg}}}(a)=0\}}\subset A\times H^0(A,\Omega_A^1),$$ which is the closure of the conormal bundle of $W_{\text{reg}}\subset A$ in $T^*A$. } 

\textbf{Step 4. We finish the proof of Proposition \ref{lem:divisorial case}.}

%{\color{red} (The idea of the proof is very similar to the Step 3 of our Theorem 4.1 Maybe we can keep both.)} 

Consider $$(T_W^*A)^o\coloneqq T^*_WA\cap (W^o\times H^0(A,\Omega_A^1)),$$ which is the conormal bundle of  $W^o$ in $T^*A$.  Since $W=f(Z)$ is of general type,  $p_2(T_W^*A)=H^0(A, \Omega_A^1)$ where $p_2\colon A\times H^0(A,\Omega_A^1)\to H^0(A,\Omega_A^1)$ is the natural projection,  by Lemma \ref{prop:DHL} and the properness of $p_2$. Clearly, $(T_W^*A)^o$ is dense in $T_W^*A$, hence $p_2((T_W^*A)^o)$ is dense in $H^0(A,\Omega_A^1)$.

Suppose by contradiction that there exists a (and hence a general) holomorphic 1-form $\omega\in H^0(A,\Omega_A^1)$ such that $f^*\omega$ has no zero. %, this is equivalent to that for a general holomorphic 1-form $\omega\in H^0(A,\Omega_A^1)$ such that $f^*\omega$ has no zero, since nowhere vanishing condition is an open condition. 
Since $p_2(T_W^*A)=H^0(A, \Omega_A^1)$, for this general holomorphic 1-form $\omega$ on $A$, one may assume there exists $(a, \omega)\in (T_W^*A)^o$ such that $\omega|_{W}(a)=0$. In other words, $\omega(a)$ kills the tangent space $T_aW$. According to Step 3, there is a projective integral curve $C_a\subset Z\cap f^{-1}(a)$, such that $$C_a\cdot Z<0, \ C_a\not\subset Z_{\text{sing}}.$$
Take any closed point $x\in C_a\cap Z_{\text{reg}}$, then $f(x)=a$.     For $W=f(Z)$,  the tangent map $df_{x}$ maps $T_xZ$ into $T_aW$. Therefore, we get that $f^*\omega(x)$ kills $T_xZ$, i.e., $f^*\omega|_{Z_{\text{reg}}}$ has zeros along $C_a\cap Z_{\text{reg}}$.

Consider the conormal sequence of $Z$ inside $X$. Since $Z$ is a divisor (hence l.c.i.) in $X$, we have short exact conormal sequence
\begin{equation}\label{conormal-sequence-divisor-case}
    \xymatrix{
	0 \ar[r] & \cN_{Z/X}^* \ar@{^{(}->}[r]^-{\theta} & \Omega^1_X|_{Z} \ar[r]^-{\delta}&\Omega_Z^1 \ar[r] &0.
}
\end{equation}
 %\[\delta\colon \omega\mapsto \omega|_X\] 
%Here $\mathcal{F}$ is the cokernel of $\theta$ and $\mathcal{F}|_{Z_{\text{reg}}}=\Omega_{Z_{\text{reg}}}^1$. 
By the preceding argument, for any closed point $x\in C_a\cap Z_{\text{reg}}$, $f^*\omega|_Z(x)=0$ in the fibre $\Omega_{Z,x}^1\otimes k(x)$, where $k(x)$ is the residue field of $\mathcal{O}_{X,x}$, which means $f^*\omega(x)$ lies in the fibre $\cN_{Z/X, x}^*\otimes k(x)$. Since $f^*\omega$ is a holomorphic 1-form without zero on $X$, we have that $f^*\omega$ gives rise to a nowhere vanishing section in $H^0(C_a\cap Z_{\text{reg}}, \cN_{Z/X}^*|_{C_a\cap Z_{\text{reg}}})$. 

Take the normalization $C$ of $C_a$ and the composite $\tau\colon C\to C_a\to Z$.  Now pull back the short exact sequence (\ref{conormal-sequence-divisor-case}) to $C$, we have the  short exact sequence 
\begin{equation}\label{conormal-sequence-divisor-case-2}
    \xymatrix{
	0\ar[r] &\tau^*\cN_{Z/X}^* \ar@{^{(}->}[r]^-{\theta'} & \tau^*\Omega^1_X|_{Z} \ar[r]^-{\delta'}&\tau^*\Omega_Z^1 \ar[r] &0,
}
\end{equation}
where the injectivity follows from that $\tau^*\cN_{Z/X}^*$ is torsion-free. Denote the saturation of the image of $\theta'$ by $L$, which is a line subbundle of $\tau^*\Omega^1_X|_{Z}$. We then have an injective morphism of rank-one locally free sheaves over $C$ $$\tau^*\cN_{Z/X}^*\hookrightarrow L.$$ Therefore we get $$L\cong \tau^*\cN_{Z/X}^*\otimes \mathcal{O}_C(\Delta)=\tau^*\mathcal{O}_Z(-Z)\otimes \mathcal{O}_C(\Delta),$$ for some effective divisor $\Delta$ on $C$. Hence $$\deg L=-Z\cdot C_a+\deg \Delta>0.$$  Since $f^*\omega$ gives rise to a nowhere vanishing section in $H^0(C_a\cap Z_{\text{reg}}, \cN_{Z/X}^*|_{C_a\cap Z_{\text{reg}}})$, hence $f^*\omega$ gives a nowhere vanishing section of $L$ over a non-empty open set of $C$.  However, since $\tau^*\Omega^1_X|_{Z}/L$ is torsion-free, $f^*\omega$ gives rise to a nonzero global section $s$ of $L$. Note that this global section $s$ is nowhere vanishing on $C$, otherwise, $f^*\omega$  would have zero on $X$. This contradicts to the fact that $\deg L>0$. The proof is done.
\end{proof}

One also needs to analyze the MMP steps involving flips. To this end, we recall some basic definitions and  introduce the following three lemmas. %{\color{red} Mention that the second part didn't use the assumption "$\exists$ holomorphic 1-form without zero and $A$ is simple"}

\begin{definition}
 Let X be a smooth projective variety, $\xi$ be a rational vector field on $X$, i.e., $\xi\in (T X)_\eta$, where $\eta$ is the generic point of $X$. At a closed point $x\in X$, we say $\xi$ is regular at $x$ if $\xi\in (TX)_x \subset (T X)_\eta $. 
 Define the pole locus of $\xi$ to be $\text{Pol}(\xi):= \{\text{closed point}\ x\in X\ |\ \xi\ \text{is not regular at}\ x\}$. Clearly, $\text{Pol}(\xi)$ is a Zariski closed subset of the variety $X$.

In general, for an algebraic variety $M$, a global holomorphic tangent vector field is a global section of the tangent sheaf $\mathcal{T}_M=  \mathscr{H}\textit{om}_{\mathcal{O}_M}(\Omega_M^1, \mathcal{O}_M)$, i.e., a $\mathbb{C}$-linear derivation of $\mathcal{O}_M$. In the case that $M$ is normal, we have $\mathcal{T}_M=j_*\mathcal{T}_{M_{\text{reg}}}$. Then every holomorphic tangent vector field of $M$ is an extension of its restriction on $M_{\text{reg}}$.
\end{definition}

\begin{lemma}\label{lem:flip-vector fields}
Let $\mu\colon X\dashrightarrow M:=Y\times A$ be a birational map from a smooth projective variety $X$ to a product of a normal projective variety $Y$ and an abelian variety $A$. Let $\xi_A$ be a global holomorphic tangent vector field on $A$. Denote  $\widetilde{\xi} \coloneqq (0, \xi_A)$ and $\xi:=\mu^*\widetilde{\xi}$. Then every irreducible component of $\textnormal{Pol}(\xi)$ is a prime divisor contracted by $\mu$.
\end{lemma}

\begin{proof}
 By \cite[Proposition~1.6]{Har80}, we get that 
  every  irreducible component of $\text{Pol}(\xi)$   must be of codimension one.
    Now suppose by contradiction that $D$ is an irreducible  component of $\text{Pol}(\xi)$  and not contracted by $\mu$.  Let $K$ be the common rational function field of $X$ and $M$. Denote $D'\coloneqq \mu_* D\subset M$, which is a prime divisor. 
Since $Y$ is normal, by Zariski main theorem, we have the following isomorphism induced by $\mu$ : $$\mathcal{O}_{X, \eta_D}\cong \mathcal{O}_{M, \eta_{D'}}.$$ where $\eta_D$ and $\eta_{D'}$ are the generic points of $D$ and $D'$, respectively. Note that  $\widetilde{\xi}$ gives rise to a derivation $d_{\widetilde{\xi}}\colon \mathcal O_{M}\to \mathcal O_{M}$, which in particular induces a derivation $d_{\widetilde{\xi}}\colon \mathcal O_{M, \eta_{D'}}\to \mathcal O_{M, \eta_{D'}}$. 
   
    Note also $d_{\widetilde{\xi}}$ is well defined on $K$ via the quotient rule of derivation. Similarly, $\xi$ induces a derivation $d_{\xi}$ on the regular functions over the Zariski open set $X-\textnormal{Pol}(\xi)$. Hence $d_{\xi}$ is well defined on $K$ and coincides with $d_{\widetilde{\xi}}$  on $K$. 
   \[\xymatrix{
\mathcal{O}_{M,\eta_{D'}}\ar[r]^-{d_{\widetilde{\xi}}}\ar@{^{(}->}[d]&\mathcal{O}_{M,\eta_{D'}}\ar@{^{(}->}[d] \\
K\ar[r]^{d_{\widetilde{\xi}}=d_{\xi}}& K.
}\] Therefore, $d_{\xi}$ maps $\mathcal{O}_{X, \eta_D}$ to $\mathcal{O}_{X, \eta_D}$, i.e., $$d_{\xi} \in \text{Der}_{\mathbb{C}}(\mathcal{O}_{X, \eta_D}, \mathcal{O}_{X, \eta_D}).$$  Since $X$ is smooth, $\text{Der}_{\mathbb{C}}(\mathcal{O}_{X, \eta_D}, \mathcal{O}_{X, \eta_D})=T_{X,\eta_D}.$ This implies $\xi$ is regular at the generic point of $D$, which is a contradiction. 
\end{proof}

\begin{lemma}\label{lem:singular-holo-flow}
Let $X$ be a quasi-projective variety, $S$ be a closed subvariety of $X$, and $\Theta$ be a global holomorphic tangent vector field on $X$ such that $\Theta(I_S)\subseteq I_S$, where $I_S$ is the ideal sheaf of $S$ in $X$. Then for any closed point $z\in S$, there exist an open neighborhood $z\in V$ in $S$, an $\varepsilon>0$, and a biholomorphism $$F_t\colon V\to V_t\coloneqq F_t(V)\subseteq S,$$ where $F_t$ is the holomorphic flow determined by the following holomorphic ODE on $X$: $$F_0=\textnormal{id}; \frac{d}{dt}F_t(x)=\Theta(F_t(x)),$$ for $|t|<\varepsilon$. 
\end{lemma}

\begin{proof}
Firstly, we assume $X$ is smooth. Fix $z\in S$. Since $I_S$ is coherent, there exists a Zariski open neighborhood $z\in W$ in $X$ such that $I_S|_W$ is generated by sections  $g_i\in \mathcal O_X(W),1\leq i\leq l$.  Since $\Theta(I_S)\subseteq  I_S$ and by shrinking $W$ if necessary,  we have $\Theta(g_i)=\Sigma_{j=1}^l a_{ij}g_j$ for  some $a_{ij}\in \mathcal O_X(W)$.

By \cite[Theorem~1.1.6 and Remark~1.1.9]{Nis09}, we know that there exists a smaller analytic open neighborhood $z\in U\subseteq W$ such that there is an $\varepsilon>0$ and a biholomorphism $F_t:U\xrightarrow{\cong} U_t:=F_t(U)\subseteq W$ for $|t|<\varepsilon$. Set $V=S\cap U$, $V_t=S\cap U_t$, then it suffices to show: \text{for any } $x\in U$,
$$x\in S\iff F_t(x)\in S.
$$
Given $x\in U$, we show that $x\in S$ implies that $F_t(x)\in S$. The argument for the other direction is exactly the same.

Denote $G_i(t)\coloneqq g_i(F_t(x)),1\leq i\leq l.$ Note that $G_i(t)$ is a differentiable function.
Then we have
\begin{align*}
  \frac{d}{dt}G_i(t) &=   \frac{d}{dt}g_i(F_t(x))\\
%&=dg_i|_{F_t(x)}(\frac{d}{dt}F_t(x))\\
%&=dg_i_{F_t(x)}\Theta(F_t)(x)\\
&=\Theta(g_i)(F_t(x))\\
%&=\Sigma _{j=1}^l( a_{ij}(F_t(x))\cdot g_j(F_t(x)))\\
&=\Sigma _{j=1}^l( a_{ij}(F_t(x))\cdot G_j(t)).
\end{align*}
Denote the $l\times l$ matrix $A(t):= (a_{ij}(F_t(x)))$ and denote $G(t) = (G_1(t),...,G_l(t))^T$. Then we have
$$
\frac{d}{dt}G(t) = A(t)G(t).
$$
Since $x\in S$, we know $G(0) = 0$. Then the uniqueness of the solution of the above ODE gives $G(t)\equiv 0 $ for $|t|<\varepsilon$, which means $F_t(x) \in S$, as desired.

For general case,  embed $X$ into a smooth algebraic variety $N$ as a closed subvariety.  Take an affine open neighborhood $N^{\circ}$ of $z$ in $N$, and $X^{\circ}\coloneqq X\cap N^{\circ}$. Then we have a short exact sequence $$0\to I_{X^{\circ}}\otimes \mathcal{T}_{N^{\circ}}\to \mathcal{T}_{N^\circ}\to \mathcal{T}_{N^{\circ}}|_{X^{\circ}}\to 0,$$ where $I_{X^{\circ}}$ is the ideal sheaf of $X^{\circ}$. Since $N^{\circ}$ is affine, we have $H^1(N^{\circ}, I_{X^{\circ}}\otimes \mathcal{T}_{N^{\circ}})=0$. Hence we have a surjection $H^0(N^{\circ}, \mathcal{T}_{N^{\circ}})\to H^0(X^{\circ},\mathcal{T}_{N^{\circ}}|_{X^{\circ}})$. Also, since $\mathcal{T}_{N^{\circ}}$ is locally free and $X^{\circ}$ closed subvariety of $N^{\circ}$, we have an injection $\mathcal{T}_{X^{\circ}}\to \mathcal{T}_{N^{\circ}}|_{X^{\circ}}$. Therefore, the vector field $\Theta|_{X^{\circ}}$ extends to a global holomorphic tangent vector field $\widetilde{\Theta}$ on $N^{\circ}$. Note that $S\cap X^{\circ}$ is a closed subvariety of $N^{\circ}$ containing $z$, and $\widetilde{\Theta}$ maps ideal sheaf of $S\cap X^{\circ}$ in $N^{\circ}$ to itself, since taking differentiation is commutative with taking restriction. Now applying the previous argument, we are done. 
\end{proof}

%The following lemma is well known to experts, we give a proof here due to the lack of appropriate references.

\begin{lemma}\label{lem:graph-iso}
Let $\mu: X\dashrightarrow M$ be a birational map between two normal projective varieties. Denote by $\Gamma$ the graph of $\mu$ and $\Gamma_W := \Gamma \cap (W\times M)$ for any (Zariski or analytic) open subset $W$ of $X$. We have the following properties: %{\color{red} simplify notation?}

%1) For any nonempty Zariski open set $W\subseteq \text{Dom}(\mu)$, define $G_W = \{ (x,\eta(x))| x\in W \} $, then $\Gamma = \overline{G_W}^{X\times M}$;

1)  For a nonempty Zariski open subset $W$ of $\textnormal{Dom}(\mu)$, let $G_W= \{ (x,\mu(x))| x\in W \}.$ Then $\Gamma _W = G_W$;

2) Let $q_X : \Gamma \to X$ be the natural projection from $\Gamma$ to $X$. Then $q_X$ is surjective, and for any (Zariski or analytic) open set $W\subseteq X$, $q_X(\Gamma _W) = W$.

%4) $q_X$ and $q_M$ are  proper birational morphisms.

\end{lemma}

\begin{proof}
%Part 1): Suppose that $W,W'\subseteq \text{Dom}(\eta)$ are two Zariski open subsets. Thus it suffices to prove $\overline{G_W}^{X\times M} = \overline{G_{W\cap W'}}^{X\times M} $. In fact, remark that $G_{W\cap W'}$ is dense in $G_W$ and   $\overline{G_{W\cap W'}}^{X\times M}\cap G_W \supseteq  G_{W\cap W'}$, so  we must have $\overline{G_{W\cap W'}}^{X\times M}\cap G_W  = G_W$ since $\overline{G_{W\cap W'}}^{X\times M}\cap G_W $ is a closed subset of $G_W$. Thus $G_W\subseteq \overline{G_{W\cap W'}}^{X\times M}\cap G_W$ and therefore $\overline{G_W}^{X\times M}\subseteq \overline{G_{W\cap W'}}^{X\times M}$.

1) %We prove if $W\subseteq \text{Dom}(\mu)$ is Zariski open, then $\Gamma_W = G_W$. In fact, 
Since $W$ is Zariski dense in $\text{Dom}(\mu)$, one has $\Gamma=\overline{G_W}^{X\times M}$, then $$\Gamma_W = \overline{G_W}^{X\times M} \cap (W\times M).$$ 
Since $G_W$ is the preimage of the diagonal $\Delta_M$ of $M\times M$ via the natural map $W\times M\to M\times M$, $G_W$ is closed in $W\times M$. Therefore, the closed points of $\overline{G_W}^{X\times M}-G_W$ are outside of $W\times M$. Hence we have $\Gamma_W=\overline{G_W}^{X\times M} \cap (W\times M) =G_W$ as desired.

2)  Note that $q_X(\Gamma)$ is closed in $X$, and it includes the dense subset $\text{Dom}(\mu)$, so we must have $q_X(\Gamma) = X$. For  any (Zariski or analytic) open set $W\subseteq X$, $\Gamma_W = \Gamma\cap (W\times M) = q_X^{-1}(W)$, so $q_X(\Gamma_W) = q_X (q_X^{-1}(W)) = W$ since $q_X$ is surjective.

%Part 4): Properness is trivial. We prove they are birational maps. Let $U$ be the maximal Zariski open subset of $X$ such that $\eta|_U:U\xrightarrow[]{\sim} U'$ is an isomorphism onto a Zariski open subset $U'\subset M$, then $\Gamma_U = G_U$ by $2)$,  and  $q_X(\Gamma_U) = q_X(G_U)\cong U, q_M(\Gamma_U) \cong  U'$.

\end{proof}

%\begin{proof}[Proof of Theorem \ref{thm:birat-zero-holomorphic 1-forms}]

\begin{theorem}\label{prop:codimension two case}
Let $f\colon X\to A$ be a morphism from a smooth projective variety  $X$ to an abelian variety $A$. Suppose there is a normal projective variety $Y$ such that $f$ is birational to the natural projection $p\colon Y\times A\to A$, i.e., there is a birational map $\mu\colon X\dashrightarrow Y\times A$ such that the following diagram commutes 
\[\xymatrix{
X\ar@{-->}[rr]^-{\mu}\ar[dr]_f&&Y\times A\ar[ld]^-p \\
&A.
}\] If $Z$ is an irreducible component of the exceptional set $\textnormal{Exc}(\mu)$ of $\mu$, and $Z$ is of  codimension at least two in $X$, then $f(Z)=A$. 
\end{theorem}

\begin{proof}
%It suffices to show that for any irreducible component $Z$ of the exceptional locus $\text{Exc}(\mu)$ of codimension at least 2, $f(Z)=A$.

Denote $M=Y\times A$ and  $m=\dim A$. Choose $m$ pointwise linearly independent global tangent vector fields $\xi_{1,A},\ldots, \xi_{m,A}$ on $A$. Define $\widetilde{\xi}_i\coloneqq (0, \xi_{i, A})$ to be the vector fields on $M$ and $\xi_i := \mu^*\widetilde{\xi}_i$ the corresponding rational vector field on $X$. Denote $D_i = \text{Pol}(\xi_i)$. By Lemma \ref{lem:flip-vector fields}, $D_i$ is a union of prime divisors of X which are all contracted by $\mu$.  Let
$\text{Exc}(\mu) =Z\cup Z_1\cup ...\cup Z_l$ be the irreducible decomposition.  Denote $$Z^\circ := Z-\cup_{i=1}^l Z_i- \cup_{i=1}^m D_i -Z_{\text{sing}}=Z-\cup_{i=1}^l Z_i-Z_{\text{sing}},$$ where the second equality holds because each irreducible component of $D_i$ is an irreducible divisorial component of $\text{Exc}(\mu)$. Note that $Z^\circ$ is a nonempty Zariski open subset of $Z$.

%\textbf{Step 1: Claim that $Z^\circ$ is a nonempty Zariski open subset of $Z$.}

%In fact, since $Z$ is irreducible, the only possible case that $Z^{\circ}$ is empty is that $Z\subseteq D$ for some irreducible component $D$ of some $D_i$. However, by Lemma \ref{lem:flip-vector fields}, $D$ is of codimension one and contracted by $\mu$, in particular, $D\subseteq$ is an irreducibel component of $\text{Exc}(\mu)$. Since $Z$ is an irreducible component of $\text{Exc}(\mu)$ of codimension at least 2  $D=Z$. This contradicts the fact that $D$ is of codimension one.

Denote the Zariski open set $N:= X-\cup_{i=1}^l Z_i- \cup_{i=1}^m D_i -Z_{\text{sing}}$, then we have $$Z^\circ  = Z\cap N= \text{Exc}(\mu)\cap N.$$ Note that $\xi_i$ is regular on $N$ for $i=1,...,m$. 
Now denote by $\xi$ one of the $\xi_i$ for simplicity.
Given any closed point $z\in Z^\circ \subseteq N$, by \cite[Theorem~1.1.6 and Remark~1.1.9]{Nis09},  there exists a small open analytic neighborhood $z\in V\subseteq N$ such that there is a biholomorphism $$
\Phi_t: V\xrightarrow{\cong}  V_t\coloneqq \Phi_t(V)\subseteq N,
$$ for sufficiently small $t$, where $\Phi_t$ is the local holomorphic flow determined by the following holomorphic ODE: $$\Phi_0=\text{id}; \frac{d}{dt}\Phi_t(x)=\xi(\Phi_t(x)).$$  To prove $f(Z) = A$, we claim that it suffices to prove for any closed point $z\in Z^\circ$, $\Phi_t(z)\in Z^\circ$ for sufficiently small t.
In fact, $f=p\circ\mu$ as morphisms on $U$, where $U$ is the maximal Zariski open subset of $X$ such that $\mu|_U:U\xrightarrow[]{\sim} U'$ is an isomorphism onto a Zariski open subset $U'\subset M$. Thus $df(\xi_i) = dp(d\mu (\xi_i)) = \xi_{i,A}$ on $U$, for $1\leq i\leq m$. Since $\xi_i$ is regular on $N$ for all $i$, and $U$ is a Zariski open dense subset of $N$, $df(\xi_i) = \xi_{i,A}$ holds in the whole $N$ for $1\leq i\leq m$.
For any closed point $z\in Z^\circ$, $\Phi_t(z)\in Z^\circ$ implies that $\xi_i(z)\in T_zZ^\circ$ for $1\leq i\leq m$. Thus we have $df(T_zZ^\circ)= T_{f(z)}A$.
It follows that $\text{dim}f(Z^\circ) = \text{dim A}$, hence $f(Z) = A$.

Now we prove the above claim. In fact, we prove a stronger result: $\Phi_t(Z\cap V) = Z\cap V_t$ for sufficiently small $t$. %, where For any closed point $z\in Z^\circ$, again denote $\xi$ to be one of the $\xi_i$ as above, we can in fact prove a stronger result: 
Note that $\Phi_t(V) = V_t$, and by our choice of $N$, $V=( U\cap V)\cup (Z\cap V)$ and $V_t=( U\cap V_t)\cup (Z\cap V_t)$. Therefore, to prove $\Phi_t(Z\cap V) = Z\cap V_t$, it suffices to prove $\Phi_t(U\cap V) = U\cap V_t$ for sufficiently small $t$, since $\Phi_t$ is a biholomorphism on $V$.

Before the proof, we fix the following notations. Denote the holomorphic vector field $\xi_A$ to be one of the tangent vector fields $\xi_{i,A}$, $\widetilde{\xi}$ to be the holomorphic vector field $(0, \xi_{A})$ on $M=Y\times A$, and again denote $\xi=\mu^*\widetilde{\xi}$ which is one of the rational vector fields $\xi_i$ on $X$.  Let $\Psi_t$ be the local holomorphic flow of $\widetilde{\xi}$ on $M$:
\begin{align*}
 \Psi_t: M &\to M,\\
(y,a) &\mapsto (y, a+\text{Exp}_A(t\xi_{A}(0))),  
\end{align*}
here $\text{Exp}_A$ is the exponential map $\text{Exp}_A: T_0A\to A$.
Let $\Theta = (\xi, \widetilde{\xi})$ be the corresponding rational vector field on $X\times M$. Since $\xi$ is regular on $N$, $\Theta$ is a holomorphic vector field on $N\times M$. Define 
\begin{align*}
 F_t: = (\Phi_t,\Psi_t) :V\times M &\to V_t\times M,\\
 (x,m) &\mapsto (\Phi_t(x),\Psi_t(m))
\end{align*}
to be the local holomorphic flow of $\Theta$.  Denote by $\Gamma$ the graph of $\mu$, and recall the notations in Lemma \ref{lem:graph-iso}:  $$\Gamma_W := \Gamma \cap (W\times M); \ G_{W'} := \{ (x,\mu(x))| x\in W' \},$$ for any (Zariski or analytic) open subset $W$ of $X$, and Zariski open subset $W'$  of $\text{Dom}(\mu)$.  Now we prove  $\Phi_t(U\cap V) = U\cap V_t$ for sufficiently small $t$ in several steps. 

\textbf{Step 1. $\Theta(I_{\Gamma_N})\subseteq I_{\Gamma_N}$, where $I_{\Gamma_N}$ is the ideal sheaf of $\Gamma_N$ in $N\times M$.}

By Lemma \ref{lem:graph-iso}, $\Gamma_U=G_U=\{(x,\mu(x))\ |\ x\in U\}$. Hence $\Theta=(\xi, \widetilde{\xi})$ is a holomorphic tangent vector field on $\Gamma_U$. For any local section $g\in I_{\Gamma_N}$  we know $\Theta(g)|_{\Gamma_U} = 0$. Since $\Gamma_U$ is a Zariski dense open subset of $\Gamma _N$, we know $\Theta (g)|_{\Gamma_N} = 0$. Moreover, since $\Gamma_N$ is also a reduced analytic variety, $\Theta(g)|_{\Gamma_N} = 0$ implies that $\Theta (g)\in I_{\Gamma_N}$.

\textbf{Step 2. By shrinking $V$ if necessary, $F_t:\Gamma_V \xrightarrow{\sim} \Gamma_{V_t}$ is biholomorphic  for sufficiently small t.}

Denote $q_X\colon \Gamma\to X; q_M\colon \Gamma\to M$ be the natural projections.  By Step 1 and Lemma \ref{lem:singular-holo-flow},   for any closed point $p\in q_X^{-1}(z)\subset \Gamma_N\subset N\times M$, there exists a small analytic open neighborhood $p\in B_p$ in $\Gamma_N$ and a small constant $\varepsilon_p>0$, such that $F_t(B_p)\subseteq \Gamma_N$ and  $F_t:B_p \xrightarrow[]{\cong} F_t(B_p)$ is biholomorphic for $|t|<\varepsilon_p$. Note that $q_X$ is proper, so $q_X^{-1}(z)$ is compact, therefore we can choose a uniform constant $\varepsilon>0$ independent of $p$ and an open analytic  neighborhood  $B \subseteq \Gamma_N$ such that $q_X^{-1}(z)\subseteq B$, 
$ F_t (B)\subseteq \Gamma_N$ and $F_t:B \xrightarrow[]{\cong} F_t(B)$ is biholomorphic for $|t|<\varepsilon$.
% By shrinking $V$ if necessary, we can assume $\Gamma_V=q_X^{-1}(V)\subseteq B$. % In fact, suppose such $V$ does not exist, then we can choose $x_n\to z$ and $\omega_n\in q_X^{-1}(x_n)-B$, by properness and after passing to a subsequence , we have $\omega_n\to \omega\in q_X^{-1}(z)\subset G$, so eventually $\omega_n\in G$ for $n$ large enough, contradiction. 
 Note that one can choose an open analytic neighborhood $z\in W$ in $X$ such that $q_X^{-1}(W)\subseteq B$. Also, by Lemma \ref{lem:graph-iso} 2), $\Gamma_W=q_X^{-1}(W)\subseteq B$.
 Shrinking $V$ if necessary, we may assume the closure $\bar{V} $ of $V$ is compact and contained in $W$, then we know that  $\Phi_t(\bar{V})\subseteq W$ for all $|t|<\varepsilon$ by decreasing the time bound $\varepsilon$ if necessary. In particular, $V_t = \Phi_t(V)\subseteq W$. By construction, we know $\Gamma_V\subseteq B$ and $\Gamma_{V_t}\subseteq B$.
 Then 
$F_t:\Gamma_V \xrightarrow{\cong} F_t(\Gamma_V)$ is biholomorphic for  $|t|<\varepsilon$. 
We claim   that $F_t(\Gamma_{V}) = \Gamma_{V_t}$. In fact, for $x\in \Gamma_V = \Gamma\cap (V\times M)$, we have  $F_t(x)\in F_t(B)\subset  \Gamma_N \subset \Gamma$. Clearly $F_t(x)\in V_t\times M$, so $F_t(\Gamma_V)\subseteq \Gamma_{V_t}$. By the same argument we have $F_{-t}(\Gamma_{V_t})\subseteq \Gamma_V$, thus $\Gamma_{V_t}\subseteq F_t(\Gamma_V)$ and the claim follows.

%{\color{red} Clearly $F_t(\Gamma_{V}) =  \Gamma_{V_t}$ by construction.  DETAILS!}

    \textbf{Step 3. $q_X\circ F_t = \Phi_t\circ q_X$ and $q_M\circ F_t = \Psi_t\circ q_M$ on $\Gamma_V$, for sufficiently small $t$.}

Denote by $\pi_X\colon X\times M\to X; \pi_M\colon X\times M\to M$ the natural projections. Clearly we have $\pi_X\circ F_t = \Phi_t\circ \pi_X$ and $\pi_M\circ F_t = \Psi_t\circ \pi_M$ on $V\times M$. By Step 2, restricting both equations to $\Gamma_V$ gives the claim.

\textbf{Step 4. $F_t(\Gamma_V\cap \Gamma_U) = \Gamma_{V_t}\cap \Gamma_U$, for sufficiently small $t$.}

Take any $w\in \Gamma_V$ and set $w_t= F_t(w)$. Then by Step 2, it suffices to prove $w\in \Gamma_U$ if and only if $w_t\in \Gamma_U$.

%By the Zariski main theorem, we know that for a proper birational morphism $q:X_1\to X_2$ where $X_1$ is irreducible and $X_2$ is normal,  $q$ is quasi-finite at $w\in W$ if and only $q$ is an isomorphism in a neighborhood of $w$. 

First we claim that  $$ w\in \Gamma_U \Longleftrightarrow 
        q_X \text{ and}\ q_M\ \text{are quasi-finite at }w. $$ According to the Zariski Main theorem, if $q_X$ and $q_M$ are both quasi-finite at $w$, then there is a Zariski open neighborhood $w\in \mathcal W\subseteq \Gamma$ such that $q_X|_{\mathcal W}$ and $ q_M|_{\mathcal W}$ are  isomorphisms. Then $q_M\circ q_X^{-1}:X\dashrightarrow M$ is a birational map such that $q_M\circ q_X^{-1}(q_X(\mathcal W))\simeq q_M(\mathcal W)$. Note that $q_M\circ q_X^{-1}=\mu$, then we must have $q_X(\mathcal W)\subseteq U$ by the definition of $U$, therefore $w\in \mathcal W\subseteq q_X^{-1}(U) = \Gamma_U$ by Lemma \ref{lem:graph-iso} 2). Conversely, if $w\in \Gamma_U$, then clearly $q_X$ and $q_M$ are both quasi-finite at $w$.

Based on the above claim, we just need to show
$$ q_X \text{ and}\ q_M\ \text{are quasi-finite at }w \Longleftrightarrow 
        q_X \text{ and}\ q_M\ \text{are quasi-finite at }w_t. $$
%We claim that 
%\begin{align*}
 %   q_X\text{ is quasi-finite at }w \iff  q_X\text{ is quasi-finite at }w_t,
%\end{align*}
%and similarly
%\begin{align*}
%    q_M\text{ is quasi-finite at }w \iff  q_M\text{ is quasi-finite at }w_t.
%\end{align*}
%We know that for a morphism of complex algebraic varities $g:Z\to S$, quasi-finitness at a point $z\in Z$ is equivalent to $\text{dim}_z g^{-1}(g(z)) = 0$. Here the dimension may equivalently be computed on the associated complex analytic fiber germ.
By Step 2 and Step 3 we know $F_t:\Gamma_V \xrightarrow{\sim} \Gamma_{V_t}$, and  $q_X\circ F_t = \Phi_t\circ q_X$ and $q_M\circ F_t = \Psi_t\circ q_M$ on $\Gamma_V$.  For either $q=q_X$ or $q=q_M$, we have 
\begin{align*}
    F_t: q^{-1}(q(w))\cap \Gamma_V &\xrightarrow[]{\cong} q^{-1}(q(w_t))\cap \Gamma_{V_t},\\
    F_t(w)&=w_t.
\end{align*}
It follows that
\begin{align*}
\text{dim}_w q^{-1}(q(w)) = \text{dim}_w (q^{-1}(q(w))\cap \Gamma_V) = \text{dim}_{w_t} (q^{-1}(q(w_t))\cap \Gamma_{V_t}) = \text{dim}_{w_t} q^{-1}(q(w_t)).    
\end{align*}

\textbf{Step 5. $\Phi_t(U\cap V) = U\cap V_t$ for sufficiently small $t$.}

By Step 4,  $F_t(\Gamma_{U\cap V}) = F_t(\Gamma_V\cap\Gamma_U)=\Gamma_{V_t}\cap \Gamma_U = \Gamma_{U\cap V_t}$.
Moreover, by Lemma \ref{lem:graph-iso} 2), we have $q_X(\Gamma_{U\cap V}) = U\cap V$ and $ q_X(\Gamma_{U\cap V_t}) = U\cap V_t$. By Step 3, we know  $q_X\circ F_t = \Phi_t\circ q_X$ on $\Gamma_V$. Then we get
\begin{align*}
   \Phi_t(U\cap V) = U\cap V_t
\end{align*}
as desired.
\end{proof}

%\begin{theorem}\label{thm:birat-zero-holomorphic 1-forms}
%Let $f\colon X\to A$ be a morphism from a smooth projective variety $X$ to an abelian variety $A$. Suppose $X$ admits a good minimal model $X^{\text{min}}$ together with a birational map $\mu\colon X\dashrightarrow X^{\text{min}}$, and  there is an irreducible component $Z$ of the exceptional locus $\textnormal{Exc}(\mu)$ of $\mu$ such that $f(Z)$ is of general type, then  for every holomorphic 1-form $\omega\in H^0(A, \Omega_A^1)$, $f^*\omega$ has zero.
%\end{theorem}

\begin{proof}[Proof of Theorem \ref{thm:birat-zero-holomorphic 1-forms}]
Suppose by contradiction that there exists a holomorphic 1-form $\omega$ on $A$ such that $f^*\omega$ is nowhere vanishing.   %Then by \cite[Theorem 1.3]{HWZ26} (also see \cite[Theorem A]{Chu26b}), there exists a finite \'etale cover $\tau\colon X'\to X$ such that $X'$ (and hence a good minimal model $X'^{\text{min}}$ of $X'$) is birational to $Y\times A'$, where $Y$ is a good minimal model and $A$ is an abelian variety.
Consider the following commutative diagram
\[\xymatrix{X\ar[rd]^-f\ar@{-->}[d]^-{\mu}\\
X^{\text{min}}\ar[d]^-{\phi}\ar[r]_-{f^{\text{min}}}& A \\
S,
}\] 
where  $\phi\colon X^{\text{min}}\to S$ is the Iitaka fibration, and $f^{\text{min}}$ is the morphism induced by $f$, given that $X^{\text{min}}$ has rational singularities. Claim that $f^{\text{min}}$ does not contract a general fibre $F$ of $\phi$ to closed points. Otherwise, $f^{\text{min}}$ factors through $\phi$ and a morphism $g\colon S\to A$. But then by \cite[Theorem 2.1]{PS14}, $f^*\omega$ has zero, which is impossible.  Since the Kodaira dimension $\kappa(F)$ of the general fibre of $\phi$ is $0$, $f^{\text{min}}(F)$ is a translate of a subtorus $B$ of $A$. Now take the Stein factorization $h^{\text{min}}\colon X^{\text{min}}\to A'$ of $X^{\text{min}}\to A\to A^{\vee}\to B^{\vee},$ where $A^{\vee}, B^{\vee}$ are dual abelian varieties. By \cite[Theorem 1.1]{HWZ26}, $A'$ is an abelian variety, and there exists an isogeny $A''\to A'$ such that $X^{\text{min}}\times_{A'}A''\cong Y\times A''$, where $Y$ is a normal projective variety with terminal singularities and semiample canonical divisor.
Moreover, we have an algebraic fibre space $h\colon X\to A'$ such that $h=h^{\text{min}}\circ \mu$, and the following commutative diagram induced by the base change along $A''\to A'$:
\[\xymatrix{
X' \ar@/_/[dd]_-{\rho}\ar@{-->}[rr]^-{\mu'}\ar[dr]_-{h'}&&Y\times A''\ar@/_/[dd]^-{\rho^{\text{min}}}\ar[ld]^-{\text{pr}_2} \\
&A''\ar@/_/[dd]\\
X\ar@{-->}[rr]^-{\mu}\ar[dr]_-{h}&&X^{\text{min}}\ar[ld]^-{h^{\text{min}}} \\
&A'.
}\] 

%let $\rho: X'\to X$ be the finite \'etale morphism obtained also by base change $A''\to A$.
We claim that if $Z\subseteq \text{Exc}(\mu)$ is an irreducible component of codimension at least 2, then $f(Z)$ is not of general type.

For the given $Z$, we can choose a codimension at least 2 irreducible component $Z'$ of $\text{Exc}(\mu')$ such that $\rho(Z') = Z$.
Denote $f':=f\circ \rho: X'\to A, f'^{\text{min}}\coloneqq f^{\text{min}}\circ \rho^{\text{min}} \colon Y\times A''\to A$.  For each closed point $y\in Y$, define the following homomorphism of abelian varieties $$\lambda_y\colon A''\to A; a\mapsto f'^{\text{min}}(y,a) - f'^{\text{min}}(y,0)$$
 By rigidity \cite[Lemma~4.9.1 and Corollary~4.9.3]{BL04}, the morphism $\lambda_y$ is independent of the choice of $y$, and thus we denote it by $\lambda$.
 Note that $\lambda$ is a not a constant map, otherwise, for any closed point $y\in Y$, the morphism $\{y\}\times A''\to X^{\text{min}}\to A\to B^{\vee}$ is a constant map. By the following commutative diagram, we then have $A''\to A'\to B^{\vee}$ is a constant map, which is impossible.

% Let $c:A\to B^\vee$ be the morphism used in the construction of $h^{\text{min}}$, and let $\tau :A''\to A'\to B^\vee$ be the composition induced by base change and Stein factorization. Clearly $\tau$ is finite \'etale morphism. Then we have following commutative diagram  
\[\xymatrix{
Y\times A'' \ar[r]^-{\cong} \ar[dr]_-{\text{pr}_2} &X^{\text{min}}\times_{A'} A'' \ar[r] \ar[d] &X^{\text{min}} \ar[r] \ar[d]^{h^{\text{min}}}
&A \ar[r] &A^\vee \ar[r] &B^\vee \\
&A'' \ar[r] &A'.\ar[urrr]}\] %In particular, $c\circ  f'^{\text{min}} = \tau\circ p $.
%Compairing this identity at $(y,a)$ and $(y,0)$, we get that $\tau(a) =  c(f^{\text{min}'}(y,a))  = c(f^{\text{min}'}(y,0))+ c(\lambda (a))$ and $\tau(0) = c(f^{\text{min}'}(y,0))$, which implies that $\tau = c\circ \lambda $. let $H:= \lambda (A'')$. Since $c(H) = \tau(A'')=B^\vee$, we know
Now we denote $H\coloneqq \lambda(A'')$, which is a positive dimensional abelian subvariety of $A$.

Denote $m\coloneqq \dim A''$. Choose $m$ pointwise-linearly independent global tangent vector fields $\xi_{1,A''}\ldots, \xi_{m,A''}$ on $A''$. Define $\widetilde{\xi}_i\coloneqq (0, \xi_{i, A''})$ to be the vector fields on $Y\times A''$ and $\xi_i := \mu'^*\widetilde{\xi}_i$ the corresponding rational vector field on $X'$.  By the proof of Theorem \ref{prop:codimension two case}, we know there exists a Zariski open subset $Z'^{\circ}\subset Z'_{\text{reg}}\subseteq Z'$ such that $\xi_i$ is regular on $Z'^{\circ}$  and $\xi_i(z)\in T_{z}Z'$ for any closed point $z\in Z'^{\circ}$ and $1\leq i\leq m$. 

Let $W=f(Z) = f'(Z')$ and $W^{\circ}$ be a nonempty open subset in $W_{\text{reg}}\cap f'(Z'^{\circ})$. %let $U=Z'^{\circ}\cap f'^{-1}(W_{\text{reg}})$, then easy to see that $f'(U)$, viewed as a subset of $W$, is  constructible  and  dense in $W$. Since $W$ is irreducible, this means we can choose a nonempty open subset $W^\circ \subseteq f'(U)$. So for
Then for a general point $w\in W^\circ$, there exists $z\in Z'^{\circ}$ such that  $w=f'(z)$ and \begin{align}\label{eq:main1.1}
    df'_z(\xi_i(z))\in df'_z (T_{z}Z')\subseteq T_wW,
\end{align}  for all $i$. Moreover, outside $\text{Exc}(\mu')$, we have \begin{align}\label{eq:main1.2}
     df'(\xi_i) = df'^{\text{min}}(0,\xi_{i,A''}) = d\lambda (\xi_{i,A''})
\end{align} for $1\leq i\leq m$.  Note that the Equation (\ref{eq:main1.2}) extends to the open set where all $\xi_i$ are regular, since $d\lambda (\xi_{i,A''})$ is a global holomorphic tangent vector field on $A$. Let $\zeta$ be any holomorphic tangent vector field on $A$ corresponding to $\zeta(0)\in T_0 H$, then we have $\zeta(w)\in T_wW$ for general $w\in W^\circ$ by (\ref{eq:main1.1}) and (\ref{eq:main1.2}).  Let $I_W$ be the ideal sheaf of $W$ in $A$, then by previous argument we know $\zeta(I_W)\subseteq I_W$. Therefore by Lemma \ref{lem:singular-holo-flow}, for any point $w\in W$, there exists an $\varepsilon_w>0$ and an open neighborhood $U_w\subseteq W$  such that $\Phi_t^\zeta(w)\in W$ for $|t|<\varepsilon_w,w\in U_w$, where $\Phi_t^\zeta$ is the local holomorphic flow induced by $\zeta$. Since $W$ is compact, we can find a uniform time bound $\varepsilon>0$,  such that if $|t|<\varepsilon$,  then $\Phi_t^\zeta(W)\subseteq W$ for all $w\in W$ and any fixed holomorphic tangent vector field $\zeta$ induced by elements in $T_0H$.  By iteration, we in fact have  $\Phi_t^\zeta(W)= W$ for all $t$. This shows that $W+H=W$, which means $W=f(Z)$ is not of general type. Then the previous claim holds true.

Therefore, every irreducible component $Z$ of the exceptional set $\textnormal{Exc}(\mu)$ such that $f(Z)$ is of general type must have codimension one. Applying Proposition \ref{lem:divisorial case} directly, we get that for every holomorphic 1-form $\omega\in H^0(A, \Omega_A^1)$, $f^*\omega$ has zero.

%Hence we have an irreducible component $Z'$ of $\text{Exc}(\mu')$ of codimension at least two. Then by Proposition \ref{prop:codimension two case}, $h'(Z')=A''$

\end{proof}

\subsection{Size of exceptional sets and space of nowhere vanishing holomorphic 1-forms}

As a consequence of Theorem \ref{thm:birat-zero-holomorphic 1-forms}, we give a bound of the dimension of the exceptional sets of  birational modifications based on the information from nowhere vanishing holomorphic 1-forms.  %, and  there is an irreducible component $Z$ of the exceptional locus $\textnormal{Exc}(\mu)$ of $\mu$ such that $f(Z)$ is of general type, then  for every holomorphic 1-form $\omega\in H^0(A, \Omega_A^1)$, $f^*\omega$ has zero.  

\begin{corollary}\label{prop:main1.1.1}
Let $f\colon X\to A$ be a morphism from a smooth projective variety $X$ to an abelian variety $A$. Suppose $X$ admits a good minimal model $X^{\text{min}}$ together with a birational map $\mu\colon X\dashrightarrow X^{\text{min}}$ such that $\textnormal{Exc}(\mu)$ is nonempty. For any irreducible component $Z$ of $\textnormal{Exc}(\mu)$, denote $B$ the maximal subtorus of $A$ such that $f(Z)$ is fibred by $B$, and $\pi\colon A\to A/B$ the quotient map. Then $$(\pi\circ f)^*H^0(A/B, \Omega_{A/B}^1)\subseteq W^{\dim X}(X),$$ where $W^{\dim X}(X)$ is the set of holomorphic 1-forms with zero on $X$.  Moreover, for any holomorphic 1-form $\omega$ on $A$, if $f^*\omega$ has no zero, then $\omega|_B\neq 0$. 
\end{corollary}

\begin{proof}
Since $B$ is the maximal subtorus such that $f(Z)$ is fibred by $B$,  $(\pi\circ f)(Z)$ is of general type in $A/B$. Applying Theorem \ref{thm:birat-zero-holomorphic 1-forms} to $\pi\circ f$ and the irreducible component $Z$ of $\text{Exc}(\mu)$, we get that  $$(\pi\circ f)^*H^0(A/B, \Omega_{A/B}^1)\subseteq W^{\dim X}(X).$$ 

Moreover, Suppose by contradiction that there is a holomorphic 1-form $\omega$ on $A$ such that $f^*\omega$ has no zero and $\omega|_B=0$, then there is a holomorphic 1-form $\eta$ on $A/B$ such that $\omega=\pi^*\eta$. Hence $f^*\omega=(\pi\circ f)^*\eta$ has no zero, which contradicts Theorem \ref{thm:birat-zero-holomorphic 1-forms}.
\end{proof}

\begin{theorem}\label{thm:compare with PS}
Let $X$  be a smooth projective variety with Albanese morphism $a\colon X\to A$. Suppose $X$ admits a good minimal model $X^{\text{min}}$ together with a birational map $\mu\colon X\dashrightarrow X^{\text{min}}$ such that $\textnormal{Exc}(\mu)$ is nonempty.  Then for any irreducible component $Z$ of $\textnormal{Exc}(\mu)$, and any subvector space $H\subseteq H^0(X, \Omega_X^1)$ whose nonzero elements are nowhere vanishing holomorphic 1-forms, we have $$\dim H\leq \min\{\dim a(Z)-\kappa(a(Z)),\  \dim X-\kappa(X)\}.$$ Moreover, we also have $\dim Z\geq \dim H+1$, and both inequalities are sharp.
\end{theorem}

\begin{proof} 
Consider the following commutative diagram of the graph $\Gamma$ of $\mu$ and its normalization $\Gamma^{\nu}$
\[\xymatrix{
&\Gamma^{\nu} \ar[ld]_p\ar[rd]^q\\
X\ar@{-->}[rr]^-{\mu}&& X^{\text{min}}.
}\] Note that since $X$ is smooth, $\text{Exc}(p)$ is pure of codimension one by \cite[Corollary 2.63]{KM98}. Also, since $\mu^{-1}$ does not contract any divisors, we have that $\text{Exc}(p)\subseteq \text{Exc}(q)$. Hence $\text{Exc}(\mu)= p(\text{Exc}(q))$. By \cite[Corollary 1.6]{HM07}, $\text{Exc}(\mu)$ is covered by rational curves. In particular, $\dim a(Z)\leq \dim Z-1$. Then by Corollary \ref{prop:main1.1.1}, we have \begin{align*}
    \kappa(a(Z))&=\dim a(Z)-\dim B\\ 
    &\leq \dim a(Z)-\dim H \\
    &\leq \dim Z-1-\dim H. 
  %  &=\dim X-\codim_XZ-\dim H-1
\end{align*}
The inequalities follow from the above discussion and \cite{PS14}.

For a boundary case of the first inequality, consider an abelian variety $A$ with $\dim A\geq 3$. Let $T$ be the complete intersection of two ample enough general  hyperplane sections of $A$. Then $K_T$ is ample by the adjunction formula. Now consider $$\mu\colon X\coloneqq \text{Bl}_{T}A \longrightarrow A.$$  By Theorem \ref{thm:birat-zero-holomorphic 1-forms}, every holomorphic 1-form on $\text{Bl}_TA$ has a zero. Hence $0=\dim H=\dim a(Z)-\kappa(a(Z))<\dim X-\kappa(X)$. For a boundary case of the second inequality, we take $X^{\text{min}}$ to be the product of a elliptic curve $E$ and a minimal surface $S$. Choose a closed point $s\in S$, then for the blowup $\mu\colon X\coloneqq \text{Bl}_{E\times\{s\}}E\times S\to X^{\text{min}}$, one has $\dim Z=\dim H+1$.
\end{proof}

\begin{remark} \label{rmk:compair-inequalities}
One should note that the boundary case  example in the proof of Theorem \ref{thm:compare with PS} does not indicate that the upper bound of Popa-Schnell inequality is weaker. For example, let $X^{\text{min}}$ be a smooth minimal projective  variety of general type, dimension at least 3 and irregularity $q(X^{\text{min}})=1$. Choose a general smooth curve $C$ in $X$ and denote the blow-up map $X\coloneqq \text{Bl}_CX^{\text{min}}\to X^\text{min}$ by $\mu$. Consider the Albanese map $a\colon X\to A$ where $A$ is an elliptic curve.  We have $0=\dim H=\dim X-\kappa(X)<\dim a(Z)-\kappa(a(Z))=1$, where $Z$ is the exceptional divisor of $\mu$.

In the case that $\textnormal{Exc}(\mu)$ is empty, i.e., $X=X^{\text{min}}$ is a smooth good minimal model, one can still get $\dim H\leq \dim X-\kappa(X)$ according to \cite{PS14}. In this case, one loses the information of $\dim H$ from the birational modification, while $H$ is fully understood in Proposition \ref{thm:linearity-zeros-gmm} in next section.
\end{remark}

%\begin{lemma}\label{lem:smooth-all-forms}
 %  Let $f\colon X\to A$ be a morphism from a smooth projective variety $X$ to an abelian variety $A$. Then $f$ is smooth, if and only if for any nonzero holomorphic 1-form $\omega\in H^0(A, \Omega_A^1)$, $f^*\omega$ has no zero.
%\end{lemma}

%\begin{proof}
%$f$ is smooth if and only if $f$ is a submersion, i.e., the tangent map of $f$ is surjective at any closed point of $X$. This is equivalent to that $f^*\omega$ has no zero for any nonzero holomorphic 1-form $\omega\in H^0(A, \Omega_A^1)$.
%\end{proof}

%Now we use Theorem \ref{thm:homological-trivial-deformation-no-blow-up-codim0}, Lemma \ref{prop:DHL} and Lemma \ref{lem:smooth-all-forms} to show the following

\subsection{Smooth morphism to simple abelian varieties and nowhere vanishing holomorphic 1-forms}

In this subsection, we give some discussions on \cite[Conjecture 1.5]{Hao24}.

\begin{conjecture}[{\cite[Conjecture 1.5]{Hao24}}]\label{conjecture-Hao}
   Let $f\colon X\to A$ be a morphism from a smooth projective variety $X$ to a simple abelian variety $A$. Then the following are equivalent:

    1) $f$ is smooth;
    
    2) there exists a holomorphic 1-form $\omega\in H^0(A, \Omega^1_A)$ such that $f^*\omega$ has no zero.  
\end{conjecture}

The following corollary is a consequence of Theorem \ref{thm:birat-zero-holomorphic 1-forms} towards Conjecture \ref{conjecture-Hao}.

\begin{corollary}\label{coro:simple-nowhere vanishing 1-form}
    Let $f\colon X\to A$ be a morphism from a smooth projective variety $X$ to a simple abelian variety $A$. Suppose $X$ admits a good minimal model $X^{\text{min}}$, and for each connected component $F$ of the general fibre of $f$ there exists an integer $m>0$ such that $|mK_F|\neq \emptyset$ and the base locus of $|mK_F|$ is of codimension at least 2 in $F$.
    
    Then the following are equivalent:

    1) $f$ is smooth;
    
    2) there exists a holomorphic 1-form $\omega\in H^0(A, \Omega^1_A)$ such that $f^*\omega$ has no zero.
\end{corollary}

\begin{proof}
We just need to show that $2)$ implies $1)$. %By assumption $2)$, we may assume $f$ has connected fibres via Stein factorization. 
%Let $f^{\text{min}}\colon X^{\text{min}}\to A$ be the morphism induced by $f$. According to the proof of Theorem \ref{thm:birat-zero-holomorphic 1-forms}, we consider the Stein factorization $X^{\text{min}}\to A'$ of $f^{\text{min}}$. Then the assumption $2)$ and \cite[Theorem 1.1]{HWZ26} implies that there exists a finite \'etale cover $A''\to A'$ such that $X^{\text{min}}\times_{A'}A''\cong Y\times A''$ with $Y$ a good minimal model. 
Using the same notations and arguments in the proof of Theorem \ref{thm:birat-zero-holomorphic 1-forms}, we have an isogeny $\tau\colon A''\to A$ between simple abelian varieties, a finite \'etale cover $\rho\colon X'\to X$ and a birational map $\mu'\colon X'\dashrightarrow Y\times A''$, where $Y$ is a normal projective variety with terminal singularities and semiample canonical divisor.  Moreover, we have $h'\colon X'\to A''$ and $$h'=\text{pr}_2\circ \mu';\ \tau\circ h'=f\circ \rho.$$ Note that $(\tau\circ h')^*\omega=(f\circ \rho)^*\omega$ is nowhere vanishing.  

Claim that $\text{Exc}(\mu')$ has no divisorial component. Suppose by contradiction that $Z$ is a divisorial irreducible component of $\mu'$. According to Theorem \ref{thm:birat-zero-holomorphic 1-forms}, $h'(Z)=A''$. Then for general fibre $F$ of $h'$, the scheme theoretic intersection $F\cap Z$ contains a nonempty divisor $D$ of $F$, which is contracted by $\mu'|_F\colon F\dashrightarrow Y$.  Since $Y$ has terminal singularities, $D$ is contained in the base locus of the linear system $|mK_F|$, which is impossible.

Again denote $m\coloneqq \dim A''$. Choose $m$ pointwise-linearly independent global tangent vector fields $\xi_{1,A''}\ldots, \xi_{m,A''}$ on $A''$. Define $\widetilde{\xi}_i\coloneqq (0, \xi_{i, A''})$ to be the vector fields on $Y\times A''$ and $\xi_i := \mu'^*\widetilde{\xi}_i$ the 
corresponding rational vector field on $X'$.  By Lemma \ref{lem:flip-vector fields}, $\xi_i$ are all global regular holomorphic tangent vector fields on $X'$. Outside $\text{Exc}(\mu')$, we have $dh'(\xi_i)=\xi_{i, A''}$. Thus the equalities hold over $A''$, in particular, $dh'_x\colon T_xX'\to T_{h'(x)}A''$ is surjective for any closed point $x\in X'$, i.e., $h'$ is smooth. Therefore, $f$ is smooth.  
\end{proof}

\begin{remark}
For $\dim X\leq 4$, Conjecture \ref{conjecture-Hao} holds (see \cite{Chu26}). Moreover, in Conjecture \ref{conjecture-Hao}, if all fibres of $f$ have at worst isolated singularities, or more generally the morphism $f$ has no blowup in codimension zero, i.e., the morphism from the relative conormal space $T^*_fX$ of $f$ to the simple abelian variety $A$ is equidimensional, Conjecture \ref{conjecture-Hao} holds, according to \cite[Proposition 3.3, 2)]{Sab85}. 
\end{remark}

Corollary \ref{coro:simple-nowhere vanishing 1-form} indicates that the essential difficulty of Conjecture \ref{conjecture-Hao} lies in the divisorial modification of general fibres of $f\colon X\to A$.  In fact, Conjecture \ref{conjecture-Hao} is equivalent to the following conjecture for very special case.

\begin{conjecture}\label{conj-Hao-New}
    Let $Y$ be a smooth good minimal model, and $A$ be a simple abelian variety. Let $C$ be a smooth projective subvariety of $Y\times A$ and $M$ be the blowup of $Y\times A$ along $C$. Denote the natural map $M\to A$ by $g$.  Then the following are equivalent:

1) $g$ is smooth;
    
    2) there exists a holomorphic 1-form $\omega\in H^0(A, \Omega^1_A)$ such that $g^*\omega$ has no zero.
\end{conjecture}

Before the proof of equivalence of Conjecture \ref{conj-Hao-New} and \cite[Conjecture 1.5]{Hao24}, we have the following simple lemma.

\begin{lemma}\label{lem:linearity-blowup}
    Let $X$ be a smooth projective variety and $C\subset X$ be a smooth projective subvariety with $\codim_XC\geq 2$. Let $\pi\colon Y\to X$ be the blow-up of $X$ along $C$. Then we have  $$W^{\dim Y}(Y) = \pi^{\ast} W^{\dim X}(X)\cup \pi^{\ast}\{\omega\in H^0(X, \Omega_X^1)\ |\ Z(\omega|_C)\neq\emptyset\}.$$
\end{lemma}

\begin{proof}
  Note that $\pi^{\ast}\omega\in H^0(Y, \Omega_Y^1)$, $\pi^{\ast}\omega$ has zeros in $\pi^{-1}(C)$ if and only if  $Z(\omega|_C)\neq\emptyset$.
Then we have  
\begin{align*}
W^{\dim Y}(Y) =&\{\pi^{\ast}\omega \in \pi^{\ast} H^0(X, \Omega_X^1)\ |\ Z(\pi^{\ast}\omega)\cap (Y-\pi^{-1}(C))\neq\emptyset\}\\
     &\cup \{\pi^{\ast}\omega \in \pi^{\ast} H^0(X, \Omega_X^1)\ |\ Z(\pi^{\ast}\omega)\cap \pi^{-1}(C)\neq\emptyset\} \\
     =& \pi^{\ast} \{\omega\ |\ Z(\omega)\cap (X-C)\neq\emptyset\}\cup  \pi^{\ast} \{\omega\ |\ Z(\omega|_C)\neq\emptyset\}  \\
     =& \pi^{\ast} W^{\dim X}(X)\cup \pi^{\ast}\{\omega\ |\ Z(\omega|_C)\neq\emptyset\}.
\end{align*}
 
\end{proof}

\begin{proposition}
Conjecture \ref{conjecture-Hao} is equivalent to Conjecture \ref{conj-Hao-New}.
\end{proposition}

\begin{proof}
  For any morphism $f\colon X\to A$ in Conjecture \ref{conjecture-Hao}, one  embeds $X$ into a smooth good minimal model $Y$. Let $C\subset Y\times A$ be the image of the injective morphism $X\to Y\times A; x\mapsto (x, f(x))$. Consider the morphism $g\colon M\coloneqq \text{Bl}_C(Y\times A)\to A$. 
  
  If $g^*\omega$ has no zero, then $\text{pr}_2^*\omega|_C$ has no zero by Lemma  \ref{lem:linearity-blowup}, where $\text{pr}_2\colon Y\times A\to A$ is the natural projection. Hence $f^*\omega$ has no zero.  It suffices to show that $f$ is smooth given that $g$ is smooth. In fact, suppose $f$ is not smooth, so is $\text{pr}_2|_C\colon C\to A$.  Then there exists a nonzero holomorphic 1-form $\eta$ on $A$ such that $\text{pr}_2^*\eta|_C$ has zero. Again by Lemma \ref{lem:linearity-blowup}, $g^*\eta$ has zero. This contradicts the smoothness of $g$.
\end{proof}

For Conjecture \ref{conj-Hao-New}, we have the following proposition.

\begin{proposition}
     Let $Y$ be a smooth good minimal model, and $A$ be a simple abelian variety. Let $C$ be a  smooth projective subvariety of $Y\times A$ and $M$ be the blowup of $Y\times A$ along $C$.  Denote the natural map $M\to A$ by $g$.  Suppose $K_C$ is semiample. Then the following are equivalent:

1) $g$ is smooth;
    
    2) there exists a holomorphic 1-form $\omega\in H^0(A, \Omega^1_A)$ such that $g^*\omega$ has no zero.  
\end{proposition}

\begin{proof}
Given $2)$, one has that $\text{pr}_2^*\omega|_C$ has no zero by Lemma \ref{lem:linearity-blowup}. Here $\text{pr}_2\colon Y\times A\to A$ is the natural projection. Then according to \cite[Theorem 1.2]{HWZ26} and simpleness of $A$, we have $\text{pr}_2|_C\colon C\to A$ is smooth. Thus $g$ is smooth.
\end{proof}

\section{(Non-)linearity of spaces of holomorphic 1-forms with zeros and Non-formal 1-forms}\label{sect:linearity-non-formal}
%\section{Criterion for detecting Hodge one-forms}

\subsection{Linearity properties of the set of holomorphic 1-forms with zeros}

Let $X$ be a smooth projective variety. For any holomorphic 1-form $\omega\in H^0(X, \Omega_X^1)$, let $Z(\omega)$ denote the set-theoretic zero loci of $\omega$ in $X$. We recall the following algebraic invariant introduced by Budur-Wang-Yoon \cite{BWY16}. 

\begin{definition}
The \textit{$i$-th lower level set} of $X$ is defined to be 
\[
W^i(X):=\{\omega\in H^0(X, \Omega_X^1)\ |\ Z(\omega)\neq \emptyset,\ \codim Z(\omega)\leq i \}.
\] 
\end{definition}

For a smooth projective variety $X$ of dimension $n$, $W^n(X)$ is the whole space of holomorphic 1-forms with zeros on $X$. Inside of $W^i(X)$, one has subsets arising from the tangent cone of the cohomology jump loci, which is very well studied,
\[
V^i(X):=\{\omega\in H^0(X, \Omega_X^1)\ |\ (H^{\bullet}(X,\CC),\wedge\omega)\text{ is not exact at } H^j(X,\CC) \text{ for some } j\leq i\}.
\]
Consider the character variety $\Hom(\pi_1(X),\CC^{\ast})$ of $X$ which is a direct product of $(\CC^{\ast})^{b_1(X)}$ with a finite group. Let $\cV^i(X)$ be the subset of $\Hom(\pi_1(X),\CC^{\ast})$ consisting of rank-one local systems $\cL$ on $X$ such that $H^{j}(X,\cL)\neq 0$ for some $j\leq i$. By the tangent cone theorem \cite[Theorem A]{DPS09}, $V^i(X)$ is the intersection of $H^0(X, \Omega_X^1)$ with the tangent cone of $\cV^i(X)$ at the trivial local system. Therefore, $V^i(X)$ is \textit{linear}, i.e., a union of finitely many vector subspaces in $H^0(X, \Omega_X^1)$ (see \cite{Bea88} and \cite{Sim93}).

One can show that $V^i(X)\subseteq W^i(X)$. In fact, suppose by contradiction that there is a holomorphic 1-form $\omega\in V^i(X)$ and $\omega\notin W^i(X)$. Then $\codim Z(\omega)>i$. By \cite[Proposition~3.4]{GL87}, $(H^{\bullet}(X,\CC),\wedge\omega)$ is exact at every degree at most $i$, which is a contradiction.  
%Observe that for any codimension $k$ degenerate 1-form $\omega$ on $X$, the sequence 
%\[
 %\rightarrow H^{i-1}(X,\CC) \xrightarrow{\wedge \omega} H^{i}(X,\CC) \xrightarrow{\wedge \omega} H^{i+1}(X,\CC)\xrightarrow{\wedge \omega} \cdots 
 %\]
 %is exact at $H^i(X,\CC)$ for $i<k$, in particular, $\omega\notin V^{k-1}$. If $\codim Z(\omega)=i$ and $\omega\notin V^k$, then above sequence is exact for all $i\leq k$.

Motivated by the linearity of $V^n(X)$, it is natural to ask whether $W^n(X)$ is linear \cite[Question 1.10]{DHL24}. There are some known cases. For example,  if $X$ is of general type or has nonzero topological/holomorphic Euler characteristic, then $W^n(X)$ is linear (see \cite{PS14} and \cite{GL87}). %For smooth projective varieties of $n\leq 3$, $W^n(X)$ is linear by \cite{HS21}.
If the Albanese map of $X$ is finite, then $W^n(X)$ is linear by \cite{DHL24}. Moreover, if $X$ is a subvariety of an abelian variety, then $W^n(X)$ is linear by \cite{Ueno75}. Based on works of \cite{Chu26b} and \cite{HWZ26}, we obtain the linearity property for any smooth good minimal models.

\begin{proposition}\label{thm:linearity-zeros-gmm}
    If $X$ is a smooth projective good minimal model (i.e., $K_{X}$ is semiample) of dimension $n$, then $$W^{n}(X)=\{\omega\in H^{0}(X,\Omega_{X}^{1})~|~Z(\omega)\neq \emptyset\}$$ is a  vector subspace of $H^{0}(X,\Omega_{X}^{1})$. 
\end{proposition}
\begin{proof}
    If every holomorphic 1-form on $X$ has zeros, then $W^{n}(X)=H^{0}(X,\Omega_{X}^{1})$. Now we may assume that $X$ admits a nowhere vanishing holomorphic 1-form. Then by \cite{Chu26b} or \cite{HWZ26}, there exists an isotrivial smooth fibration $f\colon X\to A$ where $A$ is a positive dimensional abelian variety. Moreover, there exists a finite 
    \'etale cover $A'\to A$ such that the base change of $X$ along $A'\to A$ splits, that is, $X':=X\times_{A}A'\cong Y\times A'$, where $Y$ is a smooth projective variety. Since $K_{X'}$ is semiample, it follows that $K_{Y}$ is semiample. 
    Iterating the above process, we can assume that there exists a finite \'etale cover $\pi:\tilde{X}\to X$ such that $\tilde{X}\cong \tilde{Y}\times B$, where $\tilde{Y}$ is a smooth projective variety and every holomorphic 1-form on $\tilde{Y}$ has zero, and $B$ is a positive-dimensional abelian variety. Denote by $p_{1}:\tilde{X}\to \tilde{Y}$ the projection to the first factor. Since any nonzero holomorphic 1-form on $B$ has no zeros, we have $W^{n}(\tilde{X})=p_{1}^{*}H^{0}(\tilde{Y},\Omega_{\tilde{Y}}^{1})$, which is a linear subspace of $H^{0}(\tilde{X},\Omega_{\tilde{X}}^{1})$. To show that $W^{n}(X)\subset H^{0}(X,\Omega_{X}^{1})$ is a linear subspace, it suffices to verify that $W^{n}(X)$ is closed under addition. Since $\pi$ is \'etale, for any two holomorphic 1-forms $\omega_{1},\omega_{2}\in W^{n}(X)$, $\pi^{*}(\omega_{1}),\pi^{*}(\omega_{2})\in W^{n}(\tilde{X})$. As $W^{n}(\tilde{X})$ is a linear subspace, we have $\pi^{*}(\omega_{1})+\pi^{*}(\omega_{2})=\pi^{*}(\omega_{1}+\omega_{2})\in W^{n}(\tilde{X})$. Using again that $\pi$ is \'etale we conclude $\omega_{1}+\omega_{2}\in W^{n}(X)$.   
\end{proof}

%For varieties $X$ with simple Albanese varieties, we obtain the linearity property of $X$ by Theorem \ref{thm:simple-smooth}.

%\begin{theorem}\label{prop:simple-linear}
 % Let $f\colon X\to A$ be a morphism from an $n$-dimensional smooth projective variety $X$ to a simple abelian variety $A$,  then $W^n(X)\cap f^*H^0(A,\Omega_A^1)=f^*H^0(A,\Omega_A^1)$ or $\{0\}$. In particular, if the Albanese variety of $X$ is simple, then $W^n(X)=H^0(X, \Omega_X^1)$ or $\{0\}$. 
%\end{theorem}

%\begin{proof}
%This trivially follows from Theorem \ref{thm:simple-smooth}.
%\end{proof}

\begin{example}
    In general, for an $n$-dimensional smooth projective variety $X$, $W^n(X)$ is not always a vector subspace. For example, according to \cite{SY25}, let $Y=E_{1}\times E_{2}\times Z$, where $E_{1}$, $E_{2}$ are two non-isogenous elliptic curves, $Z$ is a simply connected smooth projective surface. Let $p,q\in Z$ be two distinct closed points, and let $C_{1}=E_{1}\times\{e_{2}\}\times\{p\}$, $C_{2}=\{e_{1}\}\times E_{2}\times\{q\}$. Let $\rho:X\to Y$ be the blow-up of $Y$ along $C_{1}\cup C_{2}$. Then we have
    $$W^{4}(X)=p_{1}^{*}H^{0}(E_1,\Omega_{E_{1}}^{1})\cup p_{2}^{*}H^{0}(E_2,\Omega_{E_{2}}^{1}),$$
    where $p_{i}\colon X\to E_{i}$ $(i=1,2)$ are the natural projections. In particular, $W^{4}(X)$ is a linear subset of $H^{0}(X,\Omega_{X}^{1})$ but not a vector subspace. 
\end{example}

\begin{proposition}\label{prop:23linearity}
 Let $X$ be a smooth projective variety with $\dim X\leq 3$. Then $W^{\dim X}(X)$ is linear.   
\end{proposition}

\begin{proof}
This follows directly from Proposition \ref{thm:linearity-zeros-gmm}, Lemma \ref{lem:linearity-blowup}, \cite[Corollary 3.1]{Sch21}, and \cite[Theorem 1.3]{HS21}.
\end{proof}

\subsection{Nonlinear Examples}

In this section,  we give an example for a smooth projective subvariety of an abelian variety with a nonlinear lower level set. 

\begin{theorem}\label{thm:non-linear example}
There exists a $4$-dimensional smooth projective subvariety $X$ of an abelian 6-fold, such that the third lower level set $W^3(X)$ is not linear.
\end{theorem}
 
To deal with the above theorem, we first introduce the following lemma, which is known to experts. The proof presented here is due to Christian Schnell.

\begin{lemma}\label{lem:conormalspace-nonlinear}
  Let $X$ be a smooth projective variety, and let $f\colon X\to A$ be a morphism to an abelian variety $A$. Set $$Z_{f}=\{(x,\omega)\in X\times H^{0}(A,\Omega^{1}_{A}) ~|~~ f^{*}\omega(x)=0 \}$$ and $$S_{f}\coloneqq(f\times \mathrm{id})(Z_{f})%=\{(a,\omega)\in A\times H^{0}(A,\Omega^{1}_{A}) ~|~~  f^{*}\omega(x)=0,\ \text{for some }x\in f^{-1}(a) \},
  $$ then 
$\mathrm{dim} S_{f}\leq\mathrm{dim} A$. If $W$ is an irreducible component of $S_{f}$ of dimension $\mathrm{dim}A$, then $W$ is a conormal variety.   
\end{lemma} 

\begin{proof}
Denote by $p_{1}\colon S_{f}\to A$, $p_{2}\colon S_{f}\to H^{0}(A,\Omega^{1}_{A})$ the two natural projections. Let $Z=p_{1}(W)\subseteq A$. Note that $S_{f}$ is conical and $W$ is closed in $ A\times H^{0}(A,\Omega^{1}_{A})$, it follows that $Z$ is closed in $A$. It suffices to show that $W$ is contained in the conormal variety $ T^{*}_{Z}A$. By definition, for every pair $(a,\omega)\in W$, there is a point $x_0\in f^{-1}(a)$ such that $\omega$ vanishes on the image of the tangent map $df_{x_0}\colon T_{x_0}{X}\to T_{a}A$. Claim that for a general closed point $a \in Z_{\textnormal{reg}}$, we have $$df_{x_0}(T_{x_0}{X})\supseteq T_{a}Z.$$
In fact, we note that $Z\subseteq f(X)$. Take a resolution of singularities $\eta\colon \widetilde{f^{-1}(Z)}\to f^{-1}(Z)$ of the algebraic set $f^{-1}(Z)$. By generic smoothness, one can choose a nonempty open subset $U$ in $Z$ such that $f\circ \eta\colon (f\circ \eta)^{-1}(U)\to U$ is smooth. Denote $V\coloneqq (f\circ \eta)^{-1}(U)$. For any $a\in U$, and any $x\in f^{-1}(a)$, we take $\tilde{x}\in V$ such that $\eta(\tilde{x})=x$. Denote the natural map $V\to X$ by $\varphi$. Then we have $$T_aZ=d(f\circ\eta)_{\tilde{x}}(T_{\tilde{x}}V)=df_{x}(d\varphi_{\tilde{x}}(T_{\tilde{x}}V))\subseteq df_{x}(T_{x}{X}).$$
It follows that $(a,\omega)\in T^{*}_{Z_{\textnormal{reg}}}A$, hence $W \subseteq T^{*}_{Z}A$.
\end{proof}

\begin{proof}[Proof of Theorem \ref{thm:non-linear example}]
Take any simple abelian $5$-fold $A$ and $Y\subset A$ a general hyperplane section, in particular, $Y$ is a smooth projective $4$-fold of general type, and the normal sheaf $\cN_{Y/A}$ of $Y$ in $A$ is ample, since $\cN_{Y/A}=\cO_{A}(Y)|_{Y}$. Consider the conormal sequence \[0\to \cN^*_{Y/A}\to \mathcal{O}_Y\otimes H^0(A, \Omega_A^1)\to \Omega_Y^1\to 0\] and the short exact sequence of fibres over each closed point $y\in Y$ 
\begin{equation}\label{eq:conormal-seq}
  0\to \cN^*_{Y/A, y}\otimes k(y)\to H^0(A, \Omega_A^1)\to \Omega^1_{Y,y}\otimes k(y)\to 0  
\end{equation}
 We then have the induced Gauss map  from $Y$ to the Grassmannian \[
 \cG\colon Y\to \textnormal{Gr}(1,H^0(A,\Omega_A^1)); y\mapsto [\cN^*_{Y/A, y}\otimes k(y)].
 \]
By \cite[Corollary~2]{Ran84}, $\cG$ is a finite morphism onto its image. Moreover, by \cite[Theorem 1]{WZ25}, one can pick a general irreducible smooth projective  curve $C\subset Y$ with the property that $\cG(C)\subset \textnormal{Gr}(1,H^0(A,\Omega_A^1))\cong \PP^4$ is not contained in any hyperplane of $\textnormal{Gr}(1,H^0(A,\Omega_A^1))\cong \PP^4$ (i.e., nondegenerate), and the set \[\cS_1:=\{\omega\in H^0(A,\Omega_A^1)\ |\ \omega|_Y \text{ has isolated zeros on }C\}\cup\{0\}\] is non-linear. For the reader's convenience, we explain the non-linearity of $\cS_1$ here. In fact, by the conormal sequence (\ref{eq:conormal-seq}), $$\cS_1=\bigcup_{y\in C}\cN^*_{Y/A, y}\otimes k(y)=\bigcup_{[\cN^*_{Y/A, y}\otimes k(y)]\in \mathcal{G}(C)}\cN^*_{Y/A, y}\otimes k(y),$$ which is not linear, since $\mathcal{G}(C)$ is not contained in any hyperplane. Moreover, we observe that $\cS_1$ is irreducible and not contained in any proper vector subspace of $H^0(A,\Omega_A^1)$.

Now take any elliptic curve $E$ and consider the abelian 6-fold $E\times A$. Inside $E\times A$, we have smooth subvarieties $E\times C\subset E\times Y\subset E\times A.$ We will show that one can choose a smooth hyperplane section $X$ of $E\times Y$ that passes through $E\times C$, and the third lower level set $W^3(X)$ is not linear. We will proceed with several steps.

\textbf{Step 1. We show that there exists a smooth hyperplane section $X$ of $E\times Y$ such that $E\times C\subset X$.}

In fact, fix  a very ample line bundle $L$ on $E\times Y$, and denote by $\mathcal{I}_{E\times C}$ the ideal sheaf of $E\times C$ in $E\times Y$. We  consider the short exact sequence %(after twisted by $L^{\otimes m}$ for any integer $m>0$)
\[
0\to \cI_{E\times C}\otimes L^{\otimes m}\to\cO_{E\times Y}\otimes L^{\otimes m}\to\cO_{E\times C}\otimes L^{\otimes m}\to 0
\]
and the induced long exact sequence
\[
0\to H^0(\cI_{E\times C}\otimes L^{\otimes m})\to H^0(\cO_{E\times Y}\otimes L^{\otimes m})\to H^0(\cO_{E\times C}\otimes L^{\otimes m})\to H^1(\cI_{E\times C}\otimes L^{\otimes m})
\]
By asymptotic Riemann-Roch and Serre vanishing, the leading term of $h^0(\cI_{E\times C}\otimes L^{\otimes m})$ is $am^5$ for some $a>0$ and $H^1(\cI_{E\times C}\otimes L^{\otimes m})=0$  for sufficiently large $m$. Then for $m\gg 0$, inside of the complete linear system $|L^{\otimes m}|$, we have the linear system $\mathfrak{d}$ corresponding to $H^0(\cI_{E\times C}\otimes L^{\otimes m})$ whose scheme-theoretic base locus is $E\times C$. Then by the strong Bertini theorem \cite[Theorem 3.1]{DH91}, since $2\dim (E\times C)< \dim (E\times Y)$, there exists a hyperplane section $X$ in $\mathfrak{d}$ such that $X$ is smooth and contains $E\times C$.

\textbf{Step 2. We may assume that for the chosen hyperplane section $X$ in Step 1, the natural map $\pi\colon X\hookrightarrow E\times Y\to Y$ is finite away from $E\times C$.}

We take the embedding $E\times Y\subset \mathbb{P}^n$ given by $|L^{\otimes m}|$. For a general hyperplane section $H$ of $E\times Y$, $H$ intersects curve $E_y\coloneqq E\times \{y\}$ for any closed point $y\in Y$. Hence the natural projection $H\hookrightarrow E\times Y\to Y$ is surjective and generically finite. Denote the inclusion $E\times Y\subset \PP^n$ by $i$. Then the following morphism $$\varphi\colon E\times Y\to \PP^n\times Y; (e,y)\mapsto (i(e,y), y)$$ gives a bounded smooth family of elliptic curves in $\PP^n$ over $Y$. Then we may assume that for the chosen $m$ in Step 1, and each closed point $y\in Y$,
\begin{equation}\label{eq:vanishingH1}
 H^1(E\times Y,  \cI_{E_y}\otimes L^{\otimes m})=H^1(E\times Y,  \cI_{(E\times C)\cup E_y}\otimes L^{\otimes m})=0,   
\end{equation}
 and 
 \begin{equation}\label{eq:H0>4}
  \dim H^0(E_y, L^{\otimes m}|_{E_y})>\dim Y=4.   
 \end{equation}
For an arbitrary closed point $y\in Y-C$, consider the inclusion $$E\times C\hookrightarrow (E\times C)\cup E_y,$$ where $(E\times C)\cup E_y$ is a disjoint union. Then we have a commutative diagram
$$\xymatrix{
	0\ar[r] & \mathcal{I}_{(E\times C)\cup E_y}\otimes L^{\otimes m} \ar@{^{(}->}[d] \ar[r] &L^{\otimes m} \ar@{=}[d]\ar[r]&L^{\otimes m}|_{(E\times C)\cup E_y} \ar@{->>}[d] \ar[r]&0 \\
0\ar[r] & \mathcal{I}_{E\times C}\otimes L^{\otimes m}   \ar[r] &L^{\otimes m}  \ar[r]&L^{\otimes m}|_{E\times C} \ar[r]&0. 
}$$

Hence, by Equation (\ref{eq:vanishingH1}), we get the following commutative diagram 
$$\xymatrix{
	0\ar[r] & H^0(\mathcal{I}_{(E\times C)\cup E_y}\otimes L^{\otimes m}) \ar@{^{(}->}[d] \ar[r] &H^0(L^{\otimes m}) \ar@{=}[d]\ar[r]& H^0(L^{\otimes m}|_{(E\times C)\cup E_y}) \ar@{->>}[d] \ar[r]&0 \\
0\ar[r] & H^0(\mathcal{I}_{E\times C}\otimes L^{\otimes m})   \ar[r] & H^0(L^{\otimes m})  \ar[r]& H^0(L^{\otimes m}|_{E\times C}) \ar[r]&0. 
}$$

Hence \begin{align*}
	\dim H^0(\mathcal{I}_{E\times C}\otimes L^{\otimes m})-\dim H^0(\mathcal{I}_{(E\times C)\cup E_y}\otimes L^{\otimes m}) \\     
	  = \dim H^0(L^{\otimes m}|_{(E\times C)\cup E_y})-\dim H^0(L^{\otimes m}|_{E\times C}) \\
	  =\dim H^0(L^{\otimes m}|_{E_y})\\
	  > 4,
\end{align*}
by Equation (\ref{eq:H0>4}).

Consider the dual space $\PP^{n*}$ of hyperplanes in $\PP^n$, and the following incidence set $$\mathscr{A}\coloneqq \{(H, y)\in \PP^{n*}\times (Y-C)\ |\ E_y\subset H\},$$ together with two natural projections $p_1\colon \mathscr{A}\to \PP^{n*}$ and $p_2\colon \mathscr{A}\to Y-C$. We need to show that the intersection $p_1(\mathscr{A}) \cap \mathbb{P}H^0(\mathcal{I}_{E\times C}\otimes L^{\otimes m})$ is a countable union of proper subschemes of $\mathbb{P}H^0(\mathcal{I}_{E\times C}\otimes L^{\otimes m})$. It suffices to show that
$$\dim p_1(\mathscr{A})\cap \PP H^0(\mathcal{I}_{E\times C}\otimes L^{\otimes m})< \dim \PP H^0(\mathcal{I}_{E\times C}\otimes L^{\otimes m})$$ since $\mathscr{A}$ has a scheme structure. Consider the following diagram
\begin{equation}\label{dia:full-construction}
\xymatrix{
	&&\mathscr{A} \ar[ld]_-{p_1}\ar[rd]^-{p_2}   \\
\PP H^0(\mathcal{I}_{E\times C}\otimes L^{\otimes m})\ar@{^{(}->}[r]	&\PP^{n*} && Y-C.
}
\end{equation}
For each closed point $y\in Y-C$, the fibre $p_2^{-1}(y)$ is a linear subspace of $\PP^{n*}\times \{y\}$.  Also, \begin{align*}
\dim p_1p_2^{-1}(y)\cap \PP H^0(\mathcal{I}_{E\times C}\otimes L^{\otimes m})\\
=\dim \PP H^0(\mathcal{I}_{(E\times C)\cup E_y}\otimes L^{\otimes m})\\
< \dim \PP H^0(\mathcal{I}_{E\times C}\otimes L^{\otimes m})-4
\end{align*} Hence $$\dim p_1(\mathscr{A})\cap \PP H^0(\mathcal{I}_{E\times C}\otimes L^{\otimes m})< \dim \PP H^0(\mathcal{I}_{E\times C}\otimes L^{\otimes m}).$$ 
\vspace{0.03em}

%Now we show that $A$ is a  subscheme of $\PP^{n*}\times (Y-C)$. Since $E\times (Y-C)\to Y-C$ is a flat family of curves in $\PP^n$, $E_y$ have the same Hilbert polynomial, denoted as $\Phi$, for all $y\in Y$. Now consider the Hilbert scheme $\text{Hilb}_{\PP^n}^{\Phi}$ of curves in $\PP^n$ with Hilbert polynomial $\Phi$, together with the universal family $\cE\to \text{Hilb}_{\PP^n}^{\Phi}$. By the universal property, the flat family $E\times (Y-C)\to Y-C$ is induced from a morphism $f\colon Y-C \to \text{Hilb}_{\PP^n}^{\Phi}$. Note that the scheme-theoretic intersection $H\cap E=E$ gives rise to a subscheme structure $$\cA\coloneqq \{(H, [E])\in \PP^{n*}\times\text{Hilb}_{\PP^n}^{\Phi}  | E\subset H\}.$$ The pullback of $\cA$ to $\PP^{n*}\times Y-C$ via $f\colon Y-C\to \text{Hilb}_{\PP^n}^{\Phi}$ is $A$.

To formulate Step 3, we fix the following notations. By the Lefschetz hyperplane theorem, the inclusion $Y\subset A$ is the Albanese map of $Y$, and $X\subset E\times Y$ have the same Albanese. Hence $$H^0(X, \Omega_X^1)\cong H^0(E\times Y, \Omega_{E\times Y}^1)= q_1^*H^0(E, \Omega_E^1)\oplus q_2^*H^0(A, \Omega_A^1),$$ where $q_1\colon E\times Y\to E$ and $q_2\colon E\times Y\to Y\hookrightarrow A$ are natural maps.

\textbf{Step 3. For any holomorphic 1-form $\omega=\omega_1+\omega_2$ on $E\times Y$  where $0\neq\omega_1\in q_1^*H^0(E, \Omega_E^1)$ and $\omega_2\in q_2^*H^0(A, \Omega_A^1)$, we show that $\omega|_X$ has only isolated zeros.}

Note that the normal sheaf $$\cN_{X/E\times Y}\cong \cO_{E\times Y}(X)|_X=\cO(mL)|_X,$$ which is an ample line bundle on $X$. Also, $\omega=\omega_1+\omega_2$ is a holomorphic 1-form without zero on $E\times Y$ since $\omega_1\neq 0$.

Now suppose by contradiction that $\dim Z(\omega|_X)>0$, then there exists a projective integral curve $B$ contained in $Z(\omega|_X)$ set-theoretically. Consider the conormal sequence of $X$ inside $E\times Y$ 
	$$
\xymatrix{
	0 \ar[r] & \cN_{X/E\times Y}^* \ar[r] & \Omega^1_{E\times Y}|_{X} \ar[r]^-{\delta}&\Omega^1_{X} \ar[r] &0
}
$$ \[\delta\colon \omega\mapsto \omega|_X\] For any closed point $x\in B$, $\omega|_X(x)=0$ in the fibre $\Omega_{X,x}^1\otimes k(x)$, which means $\omega(x)$ is in the fibre $\cN_{X/E\times Y,x}^*\otimes k(x)$. Since $\omega$ is a holomorphic 1-form without zero on $E\times Y$, we have that $\omega$ gives rise to a nowhere vanishing section in $H^0(B, \cN_{X/E\times Y}^*|_{B})$. Taking the normalization $\nu_B\colon \widetilde{B}\to B$, we have $$\dim H^0(\widetilde{B}, \nu_B^*\cN_{X/E\times Y}^*|_{B})\geq \dim H^0(B, \cN_{X/E\times Y}^*|_{B})>0,$$
in particular, $\deg \nu_B^*\cN_{X/E\times Y}^*|_{B}\geq0$. Moreover, since $\cN_{X/E\times Y}$ is  an ample line bundle on $X$, $\nu_B^*\cN_{X/E\times Y}|_{B}$ is ample on $\widetilde{B}$. Then $\deg \nu_B^*\cN_{X/E\times Y}^*|_{B}<0$, which is a contradiction. Thus we are done for Step 3.
\vspace{0.5cm}

By Step 3, all holomorphic 1-forms $\eta\in W^3(X)$ are of the type $\eta=\omega|_X$, where  $\omega\in q_2^*H^0(A, \Omega_A^1)$, i.e., we have \begin{equation}\label{eq:w1-from-Y}
    W^3(X)\subseteq \pi^*H^0(Y, \Omega_Y^1)=\pi^*k^*H^0(A,\Omega_A^1),
\end{equation} where $k$ denotes the inclusion map $Y\to A$. Now we consider the following commutative diagram (\ref{dia:full-construction}) and fix several notations.
\begin{equation}\label{dia:full-construction}
\xymatrix{
	 E\times C \ar[r]^-{p'}\ar@{^{(}->}[d]_{\beta}  \ar@/_4 pc/[dd]_{\alpha} &C\ar@{^{(}->}[d]^{\gamma}\\
   E\times Y \ar[r]^-p & Y \ar@{_{(}->}[d]^k\\  
	X\ar@{->>}[ru]^{\pi}  \ar[r]_-f \ar@{_{(}->}[u]^j & A 
}
\end{equation}

By the purity theorem (\cite[Proposition~2]{Zar58}), the ramification locus of the morphism $X-(E\times C)\to Y-C$ is a divisor $D$. Denote the image of $D$ in $Y-C$ by $\Delta$. Then for any closed point $x\in D$, the tangent map $d\pi_{x}$ at $x$ is not surjective, and  $\pi$ is \'etale over $Y-(C \cup \Delta)$. By \cite[Proposition~6.3.10]{PAG2}, for any nontrivial holomorphic 1-form $\mu$ on $Y$, $\mu$ has at most isolated zeros on $Y$. It follows that for any 
nontrivial holomorphic 1-form $\eta\in W^3(X)$, $\eta|_{X-((E\times C)\cup D)}$ has at most isolated zeros.

\textbf{Step 4. For a general holomorphic 1-form $\xi$ on $A$, we show that $\pi^{\ast}(\xi|_Y)|_{X-(E\times C)}$ has only isolated zeros.}

Recall that the set $$\mathcal{S}_1:=\{\omega\in H^0(A,\Omega_A^1)\ |\ \omega|_Y \text{ has  isolated zeros on }C\}\cup\{0\}$$ is not linear, in particular, $\mathcal{S}_1\neq H^0(A,\Omega_A^1)$. Thus for a general holomorphic 1-form $\xi$ on $A$, $\xi|_Y$ has isolated zeros on $Y-C$ by \cite[Proposition~6.3.10]{PAG2}. Now suppose by contradiction that for a general holomorphic 1-form $\xi\in H^0(A, \Omega_A^1)$, $$\dim Z(\pi^{\ast}(\xi|_Y)|_{X-(E\times C)})\geq 1.$$

 Consider the morphism $f=k\circ\pi\colon X \to A$, which is generically finite onto its image $Y$ with exceptional set $E\times C$,
 $$Z_{f}=\{(x,\xi)\in X\times H^{0}(A,\Omega^{1}_{A}) ~|~~ f^{*}\xi(x)=0 \}$$ and $$S_{f}\coloneqq(f\times \mathrm{id})(Z_{f}).$$
Note that the  map $\psi\colon S_f\cap \{(A-C)\times H^0(A, \Omega_A^1)\}\to H^0(A, \Omega_A^1)$ induced by the second projection is dominant, since for a general holomorphic 1-form $\xi$ on $A$, $\xi|_Y$ has isolated zeros on $Y-C$. By contradiction hypothesis, the general fibre of $\psi$ has dimension $\geq 1$, it follows that $$\dim S_f\geq \dim S_f\cap \{(A-C)\times H^0(A, \Omega_A^1)\}\geq \dim H^0(A, \Omega_A^1)+1.$$ However, $\dim S_f\leq \dim A$ according to Lemma~\ref{lem:conormalspace-nonlinear}. This is a contradiction. We are done for Step 4.

\textbf{Step 5. We show that $W^3(X)$ is not linear.}

For any holomorphic 1-form $\omega\in \cS_1$, $\pi^{\ast}(\omega|_Y)\in W^3(X)$ since $\pi^{\ast}(\omega|_Y)$ vanishes along $E\times\{y\}$ for some point $y\in C$ such that $\omega|_Y(y)=0$. We have seen that $k^\ast\colon H^0(A, \Omega_A^1)\to H^0(Y, \Omega_Y^1)$ is an isomorphism and $\pi^{\ast}\colon H^0(Y, \Omega_Y^1)\to H^0(X, \Omega_X^1)$ is injective by the Lefschetz hyperplane theorem. Then  $\pi^{\ast}k^\ast\cS_1$ is irreducible, not linear, and $\pi^{\ast}k^\ast\cS_1\subseteq W^3(X)$.  Moreover, since $\cS_1$ is not contained in any proper vector subspace of $H^0(A,\Omega_A^1)$, $\pi^{\ast}k^\ast\cS_1$ is not contained in any proper vector subspace of $\pi^{\ast}k^\ast H^0(A,\Omega_A^1)$.

Suppose by contradiction that $W^3(X)$ is linear. By Equation (\ref{eq:w1-from-Y}) and previous discussion, we have $$\pi^{\ast}k^\ast\cS_1\subseteq W^3(X)\subseteq\pi^{\ast}k^\ast H^0(A,\Omega_A^1)=\pi^{\ast}H^0(Y,\Omega_Y^1).$$  %which is a 5-dimensional vector space. 
Therefore, we get 
\begin{equation}\label{eq:w3=wholespace}
    W^3(X)=\pi^{\ast}k^\ast H^0(A,\Omega_A^1)=\pi^{\ast}H^0(Y,\Omega_Y^1).
\end{equation} 
By Step 4 and Equation (\ref{eq:w3=wholespace}), for a general holomorphic 1-form $\xi$ on $A$, $\pi^{\ast}(\xi|_Y)$ has positive-dimensional zero in $E\times C$. 

\textbf{Claim $\spadesuit$: $\pi (Z(\pi^{\ast}\xi|_Y))\cap C\neq C$ for general $\xi$ on $A$.} Suppose by contradiction that $\pi (Z(\pi^{\ast}\xi|_Y))\cap C=C$  for general $\xi$ on $A$. Consider $$S_{f}\coloneqq(f\times \mathrm{id})(Z_{f})=\{(a,\xi)\in A\times H^{0}(A,\Omega^{1}_{A}) ~|~~  f^{*}\xi(x)=0,\ \text{for some }x\in f^{-1}(a) \},$$
and denote  by $\psi\colon S_f\to H^0(A,\Omega_A^1)$ the natural projection. By the contradiction hypothesis, we have the general fibre of $\psi$ has dimension $\geq 1$, it follows that $$\dim S_f \geq \dim H^0(A, \Omega_A^1)+1.$$ However, $\dim S_f\leq \dim A$ according to Lemma~\ref{lem:conormalspace-nonlinear}. This is a contradiction. Claim $\spadesuit$ holds true.

We proceed by analyzing the tangent maps. Taking any closed point $x=(e,y)\in E\times C$, via the commutative diagram (\ref{dia:full-construction}), we have the induced commutative diagram of tangent maps \begin{equation}\label{dia:tangent maps}
\xymatrix{
 T_eE\oplus T_yC \ar[r]^-{dp'_{x} }\ar@{^{(}->}[d]_{d\beta_{x}}  \ar@/_4 pc/[dd]_{d\alpha_{x}} &T_yC\ar@{^{(}->}[d]^{d\gamma_{y}}\\
   T_eE\oplus T_yY \ar[r]^-{dp_{x}} & T_yY \\  
	T_xX\ar[ru]_{d\pi_{x}}    \ar@{_{(}->}[u]^{dj_{x}},
}
\end{equation} 
where the inclusions $d\alpha_{x}, d\beta_{x}, d\gamma_{y}, dj_{x}$ are induced by natural inclusions, and the projections $dp'_{x}, dp_{x}$ are induced by the second projections. From commutative diagram (\ref{dia:tangent maps}), we know that the tangent map $d\pi_{x}\colon T_xX\to T_yY$ factors through $T_eE\oplus T_yY$. Since $dj_{x}$ is injective and $dp_{x}$ only kills $T_eE$, we obtain that 
\begin{equation}\label{eq:ker-3dim}
    \ker d\pi_{x}=T_eE,\ \text{and}\ \dim d\pi_{x}(T_xX)=3.
\end{equation} 
For a general nontrivial 1-form $\xi$ on $A$, $\pi^{\ast}(\xi|_Y)$ has 1-dimensional zeros of type $E_y=E\times \{y\}$ for finitely many closed points $y\in C$ by Claim $\spadesuit$, \cite[Proposition~6.3.10]{PAG2} and (\ref{eq:w3=wholespace}), i.e., there exists a nonempty Zariski open set $\mathcal{U}\subset H^0(A,\Omega_A^1)$ such that for any $\xi\in \mathcal{U}$, one has $\pi^{\ast}(\xi|Y)$ has 1-dimensional zeros of type $E_y=E\times \{y\}$ for finitely many closed points $y\in C$.  Then
 $\xi|_Y(\bigcup_{x\in E_y}d\pi_{x}(T_xX))=0.$
  According to Equation (\ref{eq:ker-3dim}), we have the following two situations. 

1) If $\bigcup_{x\in E_y}d\pi_{x}(T_xX)=T_yY$, then $\xi|_Y$ kills $T_yY$. Hence $\xi|_Y$ has isolated zeros at $y\in C$. This implies that $\xi\in\cS_1$. Denote $$\Sigma_1\coloneqq \{y\in C\ |\ \bigcup_{x\in E_y}d\pi_{x}(T_xX)= T_yY\}.$$ Then by above discussion, we have $$\{\xi\in \mathcal{U}\ |\ E_y\subseteq Z(\pi^*(\xi|_Y))\ \text{for some}\ y\in \Sigma_1\}\subseteq \cS_1.$$ Hence $$\dim \{\xi\in \mathcal{U}\ |\ E_y\subseteq Z(\pi^*(\xi|_Y))\ \text{for some}\ y\in \Sigma_1\}\leq \dim \cS_1=2.$$

2) If $\bigcup_{x\in E_y}d\pi_{x}(T_xX)\subsetneq T_yY$. Consider the exact conormal  sequence

\[
    \begin{tikzcd}
        0 \arrow[r]&\cN_{Y/A,y}^{*}\otimes k(y)\arrow[r]&\{y\}\times H^0(\Omega_A^1)\arrow[r, "\mu"]&  \Omega_{Y,y}^1\otimes k(y) \arrow[r]&0
    \end{tikzcd}
    \]
    where $\mu(\xi)=\xi|_Y$.  For any fixed $y_0\in C$ such that $\bigcup_{x\in E_{y_0}}d\pi_{x}(T_xX)\subsetneq T_{y_0}Y$, the holomorphic 1-forms $\xi$ on $A$ such that $\pi^{\ast}(\xi|_Y)$ have zeros along $E_{y_0}$ form a cone of dimension 1 or 2, since $3\leq \dim \bigcup_{x\in E_{y_0}}d\pi_{x}(T_xX)\leq 4$ by Equation (\ref{eq:ker-3dim}). 
    Denote $$\Sigma_2\coloneqq \{y\in C\ |\ \bigcup_{x\in E_y}d\pi_{x}(T_xX)\subsetneq T_yY\}.$$ Then we have $$\dim\{\xi\in \mathcal{U}\ |\ \pi^{\ast}(\xi|_Y)(E_y)=0\ \text{for some}\ y\in \Sigma_2\}\leq 3.$$
    
However, $\dim H^0(A, \Omega_A^1)=\dim \mathcal{U}=5$, the above calculations in 1) and 2) contradict the fact that for a general holomorphic 1-form $\xi$ on $A$, $\pi^{\ast}(\xi|_Y)$ has zeros along $E_y$ for some $y\in C$. We are done for the proof of this theorem.
\end{proof}

\subsection{Basic properties of non-formal 1-forms}

With the previous non-linear example, one could expect that holomorphic 1-forms that have zeros and do not come from cohomology jump loci are quite wild. We will discuss this kind of  holomorphic 1-forms in this subsection.

\begin{definition}
Let $X$ be a smooth projective variety, and $\omega$ be a holomorphic 1-form on $X$ with $\dim Z(\omega)=k$. We call  $\omega$ a \textit{non-formal}\footnote{The theory of cohomology jump loci  of variety $X$ is essentially derived from the formality property of the underlying CW complex of $X$, so we name the holomorphic 1-forms that have zeros and do not come from cohomology jump loci in this way.} holomorphic 1-form of $X$, if $\omega\in W^k(X)- V^k(X)$.
\end{definition}
 
We have the following criterion for non-formal holomorphic 1-forms, which is a direct consequence of \cite{GL87}.
\begin{proposition}\label{nonHodge}
Let $\omega$ be a holomorphic 1-form with zeros on a smooth projective variety $X$ such that $\codim_XZ(\omega)=k$. Then the following statements are equivalent :
    \begin{itemize}
        \item[(a)] $H^{0}(X,\cH^k(\Omega_X^{\bullet},\wedge\omega))=0$, where $\cH^k(\Omega_X^{\bullet},\wedge\omega)$ is the $k$-th cohomology sheaf of the Koszul complex $(\Omega_X^{\bullet},\wedge\omega)$ of sheaves of holomorphic forms;
        \item[(b)] $\omega$ is a non-formal holomorphic $1$-form.
    \end{itemize}
\end{proposition}
\begin{proof}
   Denote the Koszul complex $(\Omega_X^{\bullet},\wedge\omega) $  by $K^{\bullet}(\omega)$, and the $k$-th cohomology sheaf of $K^{\bullet}(\omega)$  by $\cH^k$. Then we have two  spectral sequences:

\[
^{'}E_{1}^{p,q} = H^{q}(X, K^{p}(\omega)) \Rightarrow \HH^{p+q}(K^{\bullet}(\omega))
\]

\[
^{''}E_{2}^{p,q} = H^{p}(X, \cH^{q}) \Rightarrow \HH^{p+q}(K^{\bullet}(\omega))
\]
Note that we have the following facts \cite[Proposition~3.7 and Lemma~3.8]{GL87}:
\begin{itemize}
    \item $\Supp \cH^{i}\subseteq Z(\omega)$ for all $i$;
    \item $^{'}E_{1}^{p,q}$ degenerates at  $E_{2}$ page;
    \item $\cH^{i}= 0$ for $i < k$.
\end{itemize} 
Then based on the spectral sequence $^{''}E_{2}^{p,q}$, we have 
\[
^{''}E_{2}^{0,k}= H^{0}(X, \cH^{k})=0\text{ if and only if }   \HH^k(K^{\bullet}(\omega))=0.
\]
From the spectral sequence of $^{'}E_{1}^{p,q}$, we infer that
 $\HH^k(K^{\bullet}(\omega))=0$ if and only if the sequence
 \[
 \cdots\rightarrow H^i(X,\Omega_{X}^{k-i-1}) \xrightarrow{\wedge \omega} H^i(X,\Omega_{X}^{k-i}) \xrightarrow{\wedge \omega} H^i(X,\Omega_{X}^{k-i+1}) \xrightarrow{\wedge \omega} \cdots 
 \]
 is exact at $H^i(X,\Omega_{X}^{k-i})$ for all $i$, via Hodge decomposition, which is also equivalent to that the following sequence
 \[
 \cdots\rightarrow H^{k-1}(X,\CC) \xrightarrow{\wedge \omega} H^{k}(X,\CC) \xrightarrow{\wedge \omega} H^{k+1}(X,\CC)\xrightarrow{\wedge \omega} \cdots 
 \]
 is exact at $H^k(X,\CC)$. Then the equivalence holds true.
\end{proof}

Based on the above cohomological criterion, we obtain the following simple lemma.
\begin{lemma}\label{lem:positivedim}
    If $\omega$ is a non-formal 1-form, then $\dim Z(\omega)>0$.
\end{lemma}
\begin{proof}
     Suppose that $\dim X=n$. In local coordinates $x_1,x_2,\cdots,x_n$, the holomorphic 1-form $\omega$ can be written into $\omega=\sum_{i=1}^{n}f_idx_i$ locally. Consider the map by wedging $\omega$ :
$$
\begin{tikzcd}
\Omega_{X}^{n-1} \arrow[r,"\wedge \omega"] & \Omega_{X}^{n} \\
\eta \arrow[r]                 & \sum_{i=1}^{n}f_idx_i\wedge\eta
\end{tikzcd}
$$
By \cite[Lemma~4.2]{Sch21}, $\cH^n(\Omega_X^{\bullet},\wedge\omega)\cong \Omega_X^n|_{Z_{\text{sch}}(\omega)}$, where $Z_{\text{sch}}(\omega)$ is the zero scheme of $\omega$. Suppose by contradiction that $\dim Z(\omega)=\dim Z_{\text{sch}}(\omega)=0$. By Proposition~\ref{nonHodge}, we would have 
$$
H^0(X, \cH^n)=H^0(Z_{\text{sch}}(\omega), \Omega_X^n|_{Z_{\text{sch}}(\omega)})=0,
$$
which is a contradiction.
\end{proof}

In the rest of this section, we provide some discussions on algebraic surfaces with non-formal 1-forms.

\begin{lemma}\label{lem:nonreduced}
    Let $S$ be a smooth projective surface, $f\colon S\to E$ be a morphism onto an elliptic curve $E$ and $\omega\in H^0(E, \Omega_E^1)$ be a nonzero holomorphic 1-form. If $f^{\ast}\omega$ is a non-formal holomorphic 1-form, then there exists a 1-dimensional irreducible component $F$ of $Z(f^*\omega)$, such that $F$ is $f$-vertical and the component of the scheme-theoretic fibre $f^{\ast}(f(F))$ supported on $F$ is non-reduced. 
\end{lemma}
\begin{proof}
    By Lemma~\ref{lem:positivedim}, there exists $F\subset Z(f^*\omega)$ with $\dim F=1$. Since $f$ is generically smooth, $F$ is $f$-vertical, say $f(F)=p$ for some closed point $p\in E$. Now consider the local homomorphism between discrete valuation rings $f^{\sharp}\colon \cO_{E,p}\to \cO_{S,\eta}$ where $\eta$ is the generic point of $F$. Fix  uniformizers $t\in \cO_{E,p}\text{ and } s\in \cO_{S,\eta}$ respectively, and denote $f^{\sharp}(t)=s^{k}$ for some integer $k$. For any closed point $x\in F$, the cotangent map $\mathfrak{m}_{p}/{\mathfrak{m}_{p}^2}\to\mathfrak{m}_{x}/{\mathfrak{m}_{x}^2}$ is zero, where $\mathfrak{m}$ denote maximal ideals. Then $\mathfrak{m}_{p}/{\mathfrak{m}_{p}^2}\to\mathfrak{m}_{\eta}/{\mathfrak{m}_{\eta}^2}$ is the zero map. Therefore, $f^{\sharp}(t)\in {\mathfrak{m}_{\eta}^2}$, that is $k\geq 2$.
\end{proof}
\begin{remark}\label{rmk:q=1}
    If $q(S)=1$, then the converse is also true. Indeed, the Albanese morphism $a$ of $S$ is a fibration over an elliptic curve $A_S$. Then for $0\neq\omega\in H^0(A_S, \Omega^1_{A_S})$ the sequence  \[
 0\to H^0(S,\CC) \xrightarrow{\wedge a^{\ast}\omega} H^1(S,\CC) \xrightarrow{\wedge a^{\ast}\omega} H^2(S,\CC) 
 \]
        is exact at $H^1(S,\CC)$, since $q(S)=1$.  Also, $\dim Z(a^{\ast}\omega)=1$. Hence $a^{\ast}\omega$ is non-formal by definition.
        
       % As for the claim. If $q(S)\geq 2$. Then the Albanese fibration $a_S\colon S\to C$ is a fibration over a smooth curve $C$ of genus $q(S)$. Then there is a holomorphic 1-form $\eta$ on $C$ which is not proportional to $\omega$, such that $\eta\wedge\omega=0$ on $C$. In this case, the sequence  \[
 %0\to H^0(S,\CC) \xrightarrow{\wedge a_S^{\ast}\omega} H^1(S,\CC) \xrightarrow{\wedge a_S^{\ast}\omega} H^2(S,\CC) 
 %\]
  %      is not exact at $H^1(S,\CC)$, which is a contradiction.
        
\end{remark}

\begin{lemma}\label{lem:exist-nonhodge+albdim1}
    Let $X$ be a smooth projective variety admitting a non-formal 1-form $\omega$ and $\alb.\dim(X)=1$, then the irregularity $q(X)=1$.
\end{lemma}
\begin{proof}
    Let $a\colon X\to A_X$ be the Albanese morphism. By assumption, we know that $a(X)$ is a smooth curve $C$ of genus $q(X)$ (see, e.g., \cite[Proposition 9.19]{Ueno75}).  Suppose by contradiction that $q(X)=g(C)>1$, then there exists another holomorphic 1-form $\omega'$ linearly independent of $\omega$ in $H^0(X,\Omega_X^1)$ such that $\omega\wedge \omega'=0$. This implies the complex $$0\to H^0(X, \mathcal{O}_X)\overset{\wedge\omega}{\longrightarrow} H^0(X,\Omega^1_X)\overset{\wedge\omega}{\longrightarrow} H^0(X,\Omega^2_X)$$ is not exact at $H^0(X,\Omega^1_X)$. However, since $\omega$ is non-formal with $d=\textnormal{codim}_X Z(\omega)\geq 1$, the above complex  is exact at $H^0(X,\Omega^1_X)$. This is a contradiction.
\end{proof}

%\begin{proposition}
%Let $S$ be a smooth projective surface with $q(S)=1$. Then $S$ admits a non-formal 1-form, if and only if  some Albanese fibre of $S$ has a non-reduced component.
%\end{proposition}

%\begin{proof}
 %   This follows directly from Lemma \ref{lem:nonreduced}, Remark \ref{rmk:q=1}, and Lemma \ref{lem:exist-nonhodge+albdim1}.
%\end{proof}

%\begin{corollary}
 %   Let $S$ be a smooth projective surface with $q(S)=1$, then the Albanese morphism $a_S$ either is smooth or has a singular fibre with a non-reduced component.
%\end{corollary}
%\begin{proof}
 %   If $a_S$ is not smooth, then $Z(f^{\ast}\omega)\neq\emptyset$. Also, the sequence \[
 %0\to H^0(S,\CC) \xrightarrow{\wedge a_S^{\ast}\omega} H^1(S,\CC) \xrightarrow{\wedge a_S^{\ast}\omega} H^2(S,\CC) 
 %\]
  %        is exact at $H^1(S,\CC)$. Hence $f^{\ast}\omega$ is non-Hodge. By Lemma~\ref{lem:nonreduced}, there is a non-reduced 1-dimensional component $F$ contained in $Z(f^{\ast}\omega)$.
%\end{proof}

%{\color{red}This corollary should hold for higher varieties.}

\begin{proposition}
    Let $S$ be a smooth minimal projective surface admitting a non-formal 1-form $\omega$ and $\alb.\dim(S)=2$. Then either $S$ is  of general type, or an elliptic quasi-bundle over an elliptic curve. 
    
    In the latter case, $q(S)=2, p_g(S)=1,\kappa(S)=1$ and $S$  is birational to a product-quotient surface $(B\times E)/G$ where $B$ is a smooth curve of genus $\geq 2$, $E$ is an elliptic curve, and $G$ is a finite group acting freely on $B\times E$ and  acting on $E$ by translations.
\end{proposition}
\begin{proof}
    Since $S$ is of maximal Albanese dimension, we have the Kodaira dimension $\kappa(S)\geq 0$ and $q(S)\geq 2$. Suppose  that $S$ is not of general type, then $K_S^2=0$. The assumption that $\omega$ is non-formal implies that the sequence \[
 0\to H^0(S,\CC) \xrightarrow{\wedge \omega} H^1(S,\CC) \xrightarrow{\wedge \omega} H^2(S,\CC) 
 \]
        is exact at $H^1(S,\CC)$. It follows that there exists $\eta\in H^0(S, \Omega_S^1)$, not proportional to $\omega$, such that $\eta\wedge\omega\neq 0$ in $H^0(S, \Omega_S^2)$. Also, there exists a 1-dimensional component $F\subset Z(\omega)$. Then we have  $F\subset Z(\eta\wedge\omega)$. We infer that $aF+M$ is linearly equivalent to $K_S$ for some effective divisor $M$ on $S$ and positive integer $a$. In this case, $0=K_S^2=aK_SF+K_SM$ and $K_S$ is nef, thus $K_SF=0$. Remark that $\kappa(S)=1$ by the existence of non-formal 1-forms and $\alb.\dim(S)=2$.  We consider separately the two cases.

        \subsubsection*{Case $F^2\geq 0$.}
       By \cite[Theorem~1]{Spu88}, $F$ is a rational multiple of a fibre of $f\colon S\to C$ and $\omega\in f^{\ast}H^0(C, \Omega_C^1)$. We infer that the arithmetic genus of $F$ is 1 and thus  $f$ is an elliptic fibration. Since $\omega$ is non-formal, then $g(C)=1$.
       Note that $2\leq q(S)\leq g(C)+1=2$ and $h^0((f_{\ast}\omega_{S/C})^{\vee})=q(S)-g(C)$, it follows that $f_{\ast}\omega_{S/C}\cong\cO_C$ due to the semi-positivity of Hodge bundle. By Noether formula and Leray spectral sequence, we have 
       \[
       \frac{1}{12}e(S)=\chi(\cO_S)=\deg (f_{\ast}\omega_{S/C})=0.
       \]
        By \cite[Lemma~1.5]{Ser91}, $f$ is an elliptic quasi-bundle over an elliptic base and all its singular fibres are of type ${}_mI_{0}$. And $S$ has numerical invariants :
        \[
        p_g(S)=1,q(S)=2,\kappa(S)=1.
        \]
        Since $f$ is isotrivial, say $E$ is a smooth fibre, there exist a smooth curve $B$ and a finite group $G$ acting on $B$ and acting freely on $E$  such that the following diagram 
\[
\begin{tikzcd}
S \arrow[rr, "\simeq"] \arrow[d,"f"] & & (B\times E)/G \arrow[d] \\
C \arrow[rr,"\simeq"] & & B/G
\end{tikzcd}
\]
commutes (see, e.g., \cite[Lemma V1.7, Proposition V1.8]{Bea}). Since $q(S)=g(B/G)+g(E/G)$, we infer that $G$ acts on $E$ by translations. It is clear that $g(B)\geq 2$ since $\kappa(S)=1$.
        \subsubsection*{Case $F^2< 0$.}
        We claim that this case cannot occur. In fact, it is clear that $F$ is a $(-2)$-curve. Since $K_S^2=0$, let $\phi\colon S\to B$ be the  elliptic Iitaka  fibration   with a general fibre $E$, then $K_S\equiv rE$ for some positive rational number $r>0$. $K_SF=0$ implies that $F$ is contained in a singular fibre of $\phi$. Since $\phi$ has a singular fibre with a $(-2)$-component $F$,  $\chi(\cO_S)>0$. We infer that $\deg(\phi_{\ast}\omega_{S/B})=\chi(\cO_S)>0$, $q(S)-g(B)=h^0((\phi_{\ast}\omega_{S/B})^{\vee})=0$ and thus $q(S)=g(B)$. It follows that $H^0(S,\Omega_S^1)\cong \phi^{\ast}H^0(B,\Omega_B^1)$ and $\alb.\dim (S)=1$, which contradicts the assumption that $\alb.\dim (S)=2$.
\end{proof} 

%We present an example for the second case, which was suggested by Akiyoshi Sannai.
%\begin{example}
%Let $C_1$ be the smooth projective curve of genus 2 locally given by $y^2 = x^6 + 1$, with the involution $\sigma(x, y) = (-x, y)$, and let $\pi\colon C_1\to  E:=C_1/\langle\sigma\rangle$ be the quotient map onto the elliptic curve $E$ that is locally given by $v^2 = u^3 + 1$, where $\pi: (x, y)\mapsto (u=x^2, v=y)$.

 %   Let $E'$ be an elliptic curve, let $t \in E'$ be a torsion-2 point  and consider the involution $\iota$ of $C_1\times E'$ given by
  %  \[
   % (c, z) \mapsto (\sigma(c), z + t).
 %   \]
  %  Let $S = (C_1 \times E')/\langle\iota\rangle$ and $\omega$ be the 1-form on S induced by $pr_1^*\pi^*(du/v)$, which is $\iota$-invariant because $\pi \circ \sigma = \pi$; equivalently, $\omega$ is the pullback of $du/v$ along the map $S\to E$, given by $[c, z] \mapsto \pi(c)$. Then $S$ is a smooth surface, of maximal Albanese dimension and has a non-formal 1-form pull-back from $E$.
%\end{example}

%references added in the ref.bib file
\bibliographystyle{plain}
\bibliography{ref}

% references added in the bibitem form
%\begin{thebibliography}{99}
%\bibitem[BL]{BL}Birkenhake, Christina; Lange, Herbert.  Complex abelian varieties. Grundlehren Math. Wiss., 302[Fundamental Principles of Mathematical Sciences]Springer-Verlag, Berlin, 2004.
%\end{thebibliography}

\end{document}